\documentclass[10pt,a4paper,reqno]{amsart}

\usepackage[utf8]{inputenc}

\usepackage{array,amsbsy,amscd,amsfonts,color,amssymb,amstext,amsmath,latexsym,amsthm,amscd,multicol,graphicx, tikz}
\usepackage{hyperref}
\usepackage[mathscr]{eucal}

\newtheorem{theorem}{Theorem}
\numberwithin{theorem}{section}
\newtheorem{proposition}[theorem]{Proposition}
\newtheorem{lemma}[theorem]{Lemma}
\newtheorem{corollary}[theorem]{Corollary}

\newtheorem{example}[theorem]{Example}
\theoremstyle{definition}
\theoremstyle{definition}
\theoremstyle{definition}
\theoremstyle{definition}
\theoremstyle{definition}
\theoremstyle{definition}

\makeatletter
\@namedef{subjclassname@2020}{%
  \textup{2020} Mathematics Subject Classification}
\makeatother

\begin{document}

\title[Iterative square roots of functions]{ Iterative square roots of functions}

\author{B.V. Rajarama Bhat}
\email{bhat@isibang.ac.in}

\author{Chaitanya Gopalakrishna}
\email{cberbalaje@gmail.com, chaitanya\_vs@isibang.ac.in}

\address{Indian Statistical Institute, Stat-Math. Unit, R V College Post, Bengaluru 560059, India}



\begin{abstract}
An iterative square root of a function $f$ is a function $g$ such that
$g(g(\cdot ))=f(\cdot )$. We obtain new characterizations for detecting the non-existence of such square roots for self-maps on arbitrary
sets. This is used to prove that continuous self-maps with no square roots are dense in the space of all continuous self-maps for various topological spaces. The spaces studied include
those that are homeomorphic to the unit cube in ${\mathbb R}^m$ and to the whole of $\mathbb{R}^m$ for every positive integer $m.$  On the other hand, we also prove that every continuous self-map of a space homeomorphic to the unit cube in $\mathbb{R}^m$ with a fixed point on the boundary can be approximated by iterative squares of continuous self-maps.
\end{abstract}

\maketitle

\section{Introduction}\label{sec1}





\baselineskip 16pt
\parskip 10pt

Given a continuous map $g: X\to X$ of a topological  space $X$
and a positive integer $n$, an {\it iterative root of order $n$} (or simply {\it $n^{th}$ root}) of
$g$ is a continuous map $f: X\to X$ such that
\begin{eqnarray}\label{01}
f^n(x)=g(x), \quad \forall x\in X,
\end{eqnarray}
where for each non-negative integer $k$, $f^k$ denote the $k^{\rm
	th}$ order  iterate of $f$ defined recursively by $f^0={\rm id}$,
the identity map on $X$, and $f^{k}=f\circ f^{k-1}$. For $n=2$,
a solution of \eqref{01} is called an {\it iterative square root}
(or simply a {\it square root}) of $g$.
The {\it iterative root problem} \eqref{01}, which is closely related to the embedding flow problem \cite{fort1955,Kuczma1968,zdun2014}, the invariant curve problem \cite{kuczma1990}, and linearization \cite{Bogatyi}, is of particular interest \cite{Baron-Jarczyk,Bodewadt,Kuczma1968,Targonski1981,Targonski1995,zdun-soalrz} both in the theory of the functional equations and in that of dynamical
systems, and has applications in statistics, signal processing, computer graphics, and information techniques \cite{Iannella,Kindermann,Martin}. Since the initial works of Babbage \cite{Babbage1815}, Abel \cite{Abel}, and K\"{o}nigs \cite{Konigs}, various researchers have paid increasing attention to the iterative root problem, and many advances have been made on its solutions for various class of maps, e.g. continuous maps on intervals and in particular those which are piecewise monotone \cite{blokh-coven,kuczma1961,LiYangZhang2008,LiZhang2018,Lin2014,Lin-Zeng-Zhang2017,LiuJarczykZhang2012,LiuLiZhang2018,LiuZhang2011,Liu-zhang2021,Zhang-zhang2007,wzhang1997},  continuous complex maps \cite{rice1980,solarz2003,solarz2006,zdun2000}, series and transseries \cite{Edgar}, continuous maps on planes and $\mathbb{R}^m$ \cite{Lesniak2002,Lesniak2010,narayaninsamy2000,Yu2021}, and set-valued maps \cite{Jarczyk-Zhang,li2009,li2012,LiuLiZhang2021}. In particular, various  properties of the iterative roots have also been studied, see for instance, approximation \cite{zhang-zhang}, stability \cite{Xu-zhang,zhang-zeng}, differentiability  \cite{Kuczzma1969,zdun-zhang, zhang1995}, and  category and measure \cite{Blokh,simon1989}.


Given a topological space $X$, let $\mathcal{C}(X)$ denote the set
of all continuous self-maps of $X$. 
For each
positive integer $n$, let $\mathcal{W}(n;X):=\{f^n: f\in
\mathcal{C}(X)\}$, the set of $n^{\rm th}$ iterates of all the maps
in $\mathcal{C}(X)$, and $\mathcal{W}(X):=\bigcup_{n=2}^{\infty}
\mathcal{W}(n;X)$. Then it is interesting to ask how big
these spaces are inside $\mathcal{C}(X).$ We need  a metric or topology on ${\mathcal C}(X)$ to make this question more precise.

Let $X$ be a compact metric space with metric $d$. Then it is natural to have the {\it supremum metric} (or {\it uniform metric})
$$\rho (f,g)=\sup \{ d(f(x), g(x)): x\in X\}$$
on $C(X)$, and this is our convention unless mentioned
otherwise. Humke and Laczkovich considered the above question when
$X=[0,1]$ with the usual metric $d(x,y)=|x-y|.$ This paper is
heavily motivated by their results. They proved in
\cite{Humke-Laczkovich1989,Humke-Laczkovich} that  $\mathcal{W}(2;[0,1])$ is not dense in $\mathcal{C}([0,1])$, $\mathcal{W}(2;[0,1])$ contains no balls of $\mathcal{C}([0,1])$ (i.e., its complement $(\mathcal{W}(2;[0,1]))^c$ is
dense in $\mathcal{C}([0,1])$), and
$\mathcal{W}(n;[0,1])$ is an analytic non-Borel subset of $\mathcal{C}([0,1])$.
Simon proved in \cite{simon1990,simon1991a,simon1991b} that  $\mathcal{W}(2;[0,1])$ is nowhere
dense in $\mathcal{C}([0,1])$ and
$\mathcal{W}([0,1])$ is not  dense in $\mathcal{C}([0,1])$.
Subsequently, nowhere denseness of $\mathcal{W}([0,1])$ in ${\mathcal
	C}([0,1])$ was proved by Blokh \cite{Blokh} using a method
different from that of Humke-Laczkovich and Simon.


It is seen that majority of the papers on iterative
roots of functions are restricted to continuous functions on
intervals of the real line. This is not surprising because the
setup is simpler to deal with. It has several advantages, including
the intermediate value theorem, monotonicity of bijections, and so
on. Extending these results to higher dimensions or more 
generic topological spaces is generally complex. We
develop a new approach (Theorem \ref{T3}) to determine the
non-existence of square roots of functions on arbitrary sets in Section \ref{sec2}. This
approach can be  used effectively to construct continuous functions on
a wide range of topological spaces, which do not even have
discontinuous square roots.
We demonstrate the procedure in Section \ref{sec3} by showing that  the continuous self-maps
without square roots are dense in the space of all continuous self-maps of the unit cube in ${\mathbb R}^m$ for the supremum metric $\rho$ (see
Theorem (\ref{Thm1})).
This is a
significant generalization of a similar result in
\cite{Humke-Laczkovich} (Corollary 5, p. 362) for intervals of the real line. We
believe that the most significant contribution of this article is
the method itself. 
The scheme can be fine-tuned to
various other settings.
We demonstrate this claim in Section \ref{sec4}  by proving that the continuous self-maps without square roots are dense in the space of all continuous self-maps of ${\mathbb R}^m$ for the compact-open topology.

The outline of the paper is as follows. In Section \ref{sec2}, we study
iterates and orbits of functions on arbitrary sets (without any
topology). Isaacs conducted a detailed investigation on the  fractional iterates of such functions in \cite{Issacs}. Here we discuss some of the basic aspects of such an approach. The main new result is Theorem \ref{T3}, which presents a variety of cases in which we can be certain that the given function does not admit any iterative square root. This will be our principal tool in Sections \ref{sec3} and \ref{sec4}. Section \ref{sec2} also
contains some background material and new results on fractional iterates of continuous self-maps on general topological spaces. In particular, Theorem \ref{connected components} shows that
in some cases, the problem of the non-existence of iterative roots of continuous maps can be reduced to elementary combinatorics.

In Section \ref{sec3}, we study iterative square roots of continuous
self-maps of the  unit cube $I^m$ in ${\mathbb R}^m.$ We find it
convenient to triangulate $I^m$ and sew together affine linear
maps piece by piece. For this purpose, we use some fundamental
notions and notations from the theory of simplicial complexes.
Given any continuous self-map of $I^m$, we approximate it with a piecewise
affine linear map, and in turn close to any such map, we
find another piecewise affine linear map with no square roots
(see Theorem \ref{Thm1}). This is the main result of this section. In
contrast, we also show that any continuous self-map of $I^m$
with a fixed point on the boundary can be approximated by iterative
squares of continuous self-maps (see Theorem \ref{boundary theorem}). A
simple example shows that  fixed points
on the boundary are not required for such approximations.

In Section \ref{sec4}, we analyze continuous self-maps of ${\mathbb R}^m$. As in Section \ref{sec3}, we prove that the continuous self-maps with no square roots are dense in the
space of all continuous self-maps in a suitable topology, namely the compact-open topology (see Theorem \ref{Thm2}). Finally, in  Section \ref{sec5},  we prove that the squares of continuous self-maps are $L^p$ dense in the space of all continuous self-maps of $I^m$, generalizing a result of \cite{Humke-Laczkovich} to higher dimensions.





\section{Functions on arbitrary sets and general topological spaces}\label{sec2}

In this section, we study iterative roots of self-maps on arbitrary sets and continuous self-maps on general topological spaces. First, we consider $X$ to be a non-empty set with no predefined topology, and prove some new results on the existence and non-existence of iterative square roots of maps in $\mathcal{F}(X)$, the set of all self-maps of $X$.   Surprisingly, some of
these results are useful even when $X$ has a topology and we seek for roots inside  ${\mathcal C}(X)$, the space of all continuous self-maps of $X$.

Henceforth,  let ${\mathbb Z}_+$ denote the set $\{0,1,2,\ldots\}$ of non-negative integers, and for
each $f\in \mathcal{F}(X)$  and $A\subseteq X$,  let  $R(f)$
denote the range of $f$ and $f|_A$ the restriction of $f$ to $A$.
Given an $f\in
\mathcal{F}(X)$, we define its {\it graph} as the directed graph $G=
(X, E)$, whose vertex set is $X$ and the edge set is $E=\{(x, f(x)):
x\in X \}$. In other words, we are looking at the same mathematical
structure from a different viewpoint. This allows us to borrow some
notions from graph theory. To begin, consider the connected
components of $G$ in the context of graph theory. We call $C$ a connected component of $X$ if it is the vertex set of some connected component of the graph of $f$. Observe that a pair
of vertices $x$ and $y$ are in the same connected component
$C\subseteq X$ if and only if there exist $m,n \in {\mathbb Z}_+$
such that $f^m(x)=f^n(y).$ The connected component of $X$ containing $x$ is called the {\it orbit} of $x$ under $f$, and is denoted by
$O_f(x)$. An extensive analysis of iterative square roots of
functions in $\mathcal{F}(X)$ can be found in \cite{Issacs}. In what
follows, we analyze them in some special cases and  present
all the results needed for our further discussions.

Let $f\in \mathcal{F}(X)$, and $G_1$ and $G_2$ be two connected components of the graph $G$ of $f$.
An {\it isomorphism} of $G_1$ and $G_2$ is a bijective function
$\phi:C_1 \to C_2$ between the vertex sets of $G_1$ and $G_2$ such
that any two vertices $x$ and $y$ are adjacent in $G_1$ if and only
if $\phi(x)$ and $\phi(y)$ are adjacent in $G_2$. By the definition
of graphs, this means that $y=f(x)$ if and only if  $\phi(y)=
f(\phi(x)).$ In other words, $G_2$ is isomorphic to $G_1$ if and
only if there exists a bijective function $\phi:C_1\to C_2$ such that $\phi
\circ f= f\circ \phi.$

\begin{proposition}\label{two components}
	Let $f\in \mathcal{F}(X)$ be  such that the
	graph of $f$ has exactly two isomorphic connected components. Then
	$f$ has an iterative square root in $\mathcal{F}(X)$. 
\end{proposition}
\begin{proof}
	Let $C_1$ and $C_2$ be the two isomorphic connected components of $X$ corresponding to those of the graph of $f$. Then $X=C_1\cup C_2,$ $C_1\cap C_2=\emptyset$, and as
	indicated above, there exists a bijective function $\phi:C_1\to C_2$
	such that $\phi\circ f= f\circ \phi$. Define a map $g:X\to X$ by
	\begin{eqnarray*}
		g(x) = \left\{ \begin{array}{cl}
			\phi(x)& ~\mbox{if}~x\in C_1,\\
			f\circ \phi^{-1}(x)& ~\mbox{if}~x\in C_2.\end{array}\right.
	\end{eqnarray*}
	Now, we have  $g(x)=\phi(x)\in C_2$  for each $x\in C_1$, implying that $g^2(x)=
	f\circ \phi^{-1}\circ \phi(x)= f(x)$. Similarly,  we have $g(x)=f\circ \phi^{-1}(x) \in C_1$ for each $x\in
	C_2$, and so
	$g^2(x)= \phi\circ f\circ \phi^{-1}(x)= f(x)$.
	Therefore $g$ is an iterative square root of $f$ on $X$.
\end{proof}
The above proposition leads to the following result, where a fixed point $x$ of $f$
is said to be {\it isolated} if there is no $y\ne x$ in $X$ such that
$f(y)=x$. By convention,  the empty set and all infinite sets are
assumed to have an even number of elements here and elsewhere.
\begin{theorem}\label{many components}
	Let $f\in \mathcal{F}(X)$ be such that excluding isolated fixed
	points of $f$, the number of isomorphic copies is even for each
	connected component of the graph of $f$. Then $f$ has an iterative
	square root in $\mathcal{F}(X)$.
\end{theorem}

\begin{proof} Let $g:X\to X$ be a map defined as $g(x)=x$ for isolated
	fixed points $x\in X$ of $f$ and $g$ as in the previous proposition on the union of pairs of connected components of $X$ corresponding to  pairs of
	isomorphic copies of those of the graph of $f$. Since there are an even
	number of connected components, we can pair them, and
	the result follows from the previous proposition.
\end{proof}

The above results are given for general $f\in \mathcal{F}(X)$. If we restrict ourselves to injective maps, we can describe the functions with square roots transparently, as seen below. To make our thoughts more concrete, we borrow some ideas and terminologies from \cite{BDR}.
Note that given an injective map $f\in \mathcal{F}(X)$, any two
points $x, y \in X$ are in the same orbit under $f$ if and only if there
exists an $n\in {\mathbb Z}_+$ such that either $y=f^n(x)$ or
$x=f^n(y).$ Indeed, an orbit under $f$ has one of the following
forms for some $x\in X$:
\begin{description}
	\item[{\rm (i)}] $\{x, f(x), f^2(x), \ldots,f^{d-1}(x)\}$ for some  $d\in \mathbb{N}$ with $f^d(x)=x$;
	\item[{\rm (ii)}] $\{x, f(x), f^2(x), \ldots\}$ with $x\notin f(X)$;
	\item[{\rm (iii)}] $\{\ldots ,f^{-2}(x), f^{-1}(x), x, f(x), f^2(x), \ldots\}$.
\end{description}
Therefore it follows that an  injective map $f\in
\mathcal{F}(X)$ with exactly one orbit $O_f(x)$ for some $x\in X$ is
in bijective correspondence with precisely one of the following
maps (see \cite{BDR}):
\begin{description}
	\item[{\rm (a)}] {\it Cyclic permutations}: For $d\in {\mathbb N}$, consider
	${\mathbb Z}_d =\{0, 1, 2, \ldots , d-1\}$ with addition modulo $d.$
	Define $s_d: {\mathbb Z}_d\to {\mathbb Z}_d$ by
	$s_d(k)=k+1~(\mbox{mod}~d)$. Note that ${\mathbb Z}_1=\{0\}$ and
	$s_1(0)=0$. Then $s_d$ is bijective and has exactly one orbit.
	\item[{\rm (b)}] {\it Unilateral translation/shift}: Define $s_+: {\mathbb
		Z}_+\to {\mathbb Z}_+$ by $s_+(k)=k+1.$ Then $s_+$ is a shift with a
	single orbit.
	\item[{\rm (c)}]  {\it Bilateral translation}: Define $s: {\mathbb Z}\to
	{\mathbb Z}$ by $s(k)=k+1$. Then $s$ is bijective and has exactly one
	orbit.
\end{description}
More precisely, if $f\in \mathcal{F}(X)$ has  exactly one orbit
$O_f(x)$, then there exists a bijective function $\phi: O_f(x)\to Z$
defined by $\phi(f^k(x))=k$, where $Z$ is $\mathbb{Z}_d$ (for some
$d\in \mathbb{N}$), $\mathbb{Z}_+$ or $\mathbb{Z}$ according as
$O_f(x)$ has the form  (i),  (ii) or  (iii),
respectively.  Consider any arbitrary non-empty sets $M_d$'s, $M_+,$ and
$M.$ We define $1_{M_d}\times s_d: M_d\times {\mathbb Z}_d\to
M_d\times {\mathbb
	Z}_d$ by
\begin{eqnarray*}
	(1_{M_d}\times s_d)(l,k)= (l, k+1({\rm mod}~d)),
\end{eqnarray*}
called the {\it $d$-cyclic
	permutation with multiplicity $M_d$} for each $d\in \mathbb{N}$.
Similarly, we define $1_{M_+}\times s_+:M_+\times \mathbb{Z}_+\to M_+\times \mathbb{Z}_+$ and $1_{M}\times s:M\times \mathbb{Z}\to M\times \mathbb{Z}$ by
\begin{eqnarray*}
	(1_{M_+}\times s_+)(l,k)= (l, k+1)\quad \text{and}\quad (1_M\times s)(l,k)= (l, k+1),
\end{eqnarray*}
called the {\it unilateral shift with multiplicity $M_+$} and the  {\it bilateral translation with multiplicity
	$M$}, respectively.  Then given any injective function $f\in \mathcal{F}(X)$, decomposing $X=\bigsqcup_{x\in X} O_f(x)$ into orbits under $f$ we see that $(X,f)$ is in bijective correspondence with $(Y,\tilde{f})$, where $Y= \bigsqcup _{d\in {\mathbb
		N}}(M_d\times {\mathbb Z}_d)\bigsqcup (M_+\times {\mathbb Z}_+)\bigsqcup
(M\times {\mathbb Z})$ for some suitable multiplicity spaces
$M_d$, $M_+$ and $M$ (some of these sets
may be absent in the union), and $\tilde{f}$ is $1_{M_d}\times s_d$,
$1_{M_+}\times s_+$ and $1_M\times s$ on $M_d\times {\mathbb Z}_d$,
$M_+\times {\mathbb Z}_+$ and $M\times {\mathbb Z}$, respectively.
Furthermore,  the cardinalities of these multiplicity spaces are uniquely
determined. We call $(m_1, m_2, \ldots ,  m_+,m)$ as the
multiplicity sequence of $f$, where $m_d$, $m_+$ and $m$ are the
cardinalities of $M_d$ (for $d\in {\mathbb N}$),  $M_+$ and $M$,
respectively.

\begin{theorem}\label{T2A}
	Let $f\in \mathcal{F}(X)$ be an injective map. Then $f$ has an
	iterative  square root in $\mathcal{F}(X)$ if and only if the
	multiplicity sequence $(m_1, m_2, \ldots ,  m_+,m)$ of $f$ satisfies
	that $m_d$ with $d$ even, $m_+$ and $m$ are even ($0$ and infinity
	allowed).
\end{theorem}
\begin{proof}
	Let $f\in \mathcal{F}(X)$ be an injective map and $(m_1, m_2, \ldots ,  m_+,m)$ be the multiplicity sequence of $f$.
	Suppose that $f=g^2$ for some $g\in \mathcal{F}(X)$.  Then, clearly $g$ is an  injective map on $X$, and therefore we can associate a multiplicity sequence $(m_1', m_2', \ldots ,  m_+',m')$ for $g$. Further,
	it is
	easily seen that an orbit under $g$  corresponding to a cyclic permutation on ${\mathbb Z}_d$  gives rise to two orbits (resp. a unique orbit)  under $g^2$ corresponding to cyclic permutation on ${\mathbb Z}_{\frac {d}{2}}$ (resp. on $\mathbb{Z}_d$) whenever  $d$ is even (resp. odd). Similarly, an orbit under $g$  corresponding to the unilateral shift on ${\mathbb Z}_+$ (resp. the bilateral translation on $\mathbb{Z}$) gives rise to two orbits   under $g^2$ corresponding to the unilateral shift on ${\mathbb Z}_+$ (resp. the bilateral translation on $\mathbb{Z}$).  Therefore the multiplicity sequence of $g^2$ is $(m_1'+2m_2', 2m_4', m_3'+2m_6', 2m_8',
	m_5'+2m_{10}', 2m_{12}', \ldots, 2m_+', 2m')$. Consequently, we must have $m_+=2m_+'$, $m=2m'$, and
	\begin{eqnarray*}
		m_d = \left\{ \begin{array}{cl}
			m_d'+2m_{2d}'&~\mbox{if}~d~\text{is odd},\\
			2m_{2d}'& ~\mbox{if}~d~\text{is even},\end{array}\right.
	\end{eqnarray*}
	and hence $m_d$ for $d$ even,  $m_+$ and $m$ are even.

	Conversely, suppose that the multiplicity sequence $(m_1, m_2,
	\ldots, m_+, m)$ of $f$ satisfies that $m_d$ for $d$ even,  $m_+$
	and $m$ are even. Then $X$ can be decomposed into orbits under $f$
	as $X=\bigsqcup_{x\in X} O_f(x)$ and $(X,f)$ is in bijective
	correspondence with $(Y,\tilde{f})$, where $Y= \bigsqcup _{d\in
		{\mathbb
			N}}(M_d\times {\mathbb Z}_d)\bigsqcup (M_+\times {\mathbb Z}_+)\bigsqcup
	(M\times {\mathbb Z})$ such that
	\begin{enumerate}
		\item[$\bullet$] each orbit $O_f(x)$ in $X$ under $f$ has one of the above-mentioned forms  (i),  (ii) or  (iii),
		\item[$\bullet$] the cardinalities of multiplicity spaces $M_d$ (for all $d\in \mathbb{N}$),
		$M_+$ and $M$ are $m_d$, $m_+$ and $m$, respectively, and
		\item[$\bullet$]
		$\tilde{f}$ is $1_{M_d}\times s_d$,  $1_{M_+}\times
		s_+$ and $1_M\times s$ on $M_d\times {\mathbb Z}_d$, $M_+\times {\mathbb Z}_+$ and $M\times {\mathbb Z}$, respectively.
	\end{enumerate}
	Now, define a function $g:X\to X$ as follows. If $d\in \mathbb{N}$ is odd, then for each of the orbit $O_f(x)=\{x,f(x),f^2(x),\ldots, f^{d-1}(x)\}$ corresponding to the cyclic permutation on ${\mathbb Z}_d$, define $$g(f^l(x))=f^{l+\frac{d+1}{2}(\text{mod}~d)}(x)$$ for all $0\le l\le d-1$. Then $$g^2(f^l(x))=g(f^{l+\frac{d+1}{2}(\text{mod}~d)}(x))=f^{l+d+1(\text{mod}~d)}(x)=f^{l+1}=f(f^l(x))$$ for all $0\le l\le d-1$, implying that $g^2=f$ on $O_f(x)$. If $d\in \mathbb{N}$ is even, then as $m_d$ is even, by  pairing any two distinct orbits $O_f(x)=\{x,f(x),f^2(x),\ldots, f^{d-1}(x)\}$ and $O_f(y)=\{y,f(y),f^2(y),\ldots, f^{d-1}(y)\}$ corresponding to the cyclic permutation on ${\mathbb Z}_d$, define $g$ on $O_f(x)\cup O_f(y)$ by
	\begin{eqnarray*}
		g(z) = \left\{ \begin{array}{cl}
			\phi(z)& ~\mbox{if}~z\in O_f(x),\\
			f\circ \phi^{-1}(z)& ~\mbox{if}~z\in O_f(y),\end{array}\right.
	\end{eqnarray*}
	where  $\phi:O_f(x)\to O_f(y)$ is a bijective
	function such that $\phi\circ f= f\circ \phi$. Then, by a similar argument as in Proposition  \ref{two components}, it follows that $g^2=f$ on $O_f(x)\cup O_f(y)$. Since $m_+$ (resp. $m$) is even, $g$ can be defined similarly on the union $O_f(x)\cup O_f(y)$ of each pair of distinct orbits $O_f(x)$ and $O_f(y)$ corresponding to the unilateral shift on ${\mathbb Z}_+$ (resp. the bilateral translation on $\mathbb{Z}$) such that $g^2=f$ on $O_f(x)\cup O_f(y)$.  Therefore $f$ has a square root in $\mathcal{F}(X)$.
\end{proof}

For each $f\in \mathcal{F}(X)$ and $x\in X$, let $f^{-1}(x)$ and $f^{-2}(x)$ denote the usual inverse images defined by $f^{-1}(x)=\{y\in X:
f(y)=x\}$ and $f^{-2}(x)=\{y\in X: f^2(y)=x\}$. Further, for each set $A$, let
$\# A$ denote the number of elements or the cardinality of $A$.
So far, we have studied various conditions under which maps
in $\mathcal{F}(X)$ have square roots. We now have some instances in which they
have no square roots. The following result is very useful for constructing functions without square roots in the next section.

\begin{theorem}\label{T3}
	Let $f\in \mathcal{F}(X)$ be such that $f(x_0)\ne x_0$ for some $x_0\in X$.
	Then $f$ has no iterative square roots in $\mathcal
	{F}(X)$ in the following cases:
	\begin{description}
		\item[Case (i)]  $\# f^{-2}(x_0)>1$, and $\# f^{-1}(x)\leq 1$ for all $x\neq
		x_0$;
		
		\item[Case (ii)] $ f^{-2}(x_0)$ is infinite, and $ f^{-1}(x)$ is finite for
		all $x\ne x_0$;
		
		\item[Case (iii)]  $f^{-2}(x_0)$ is uncountable, and $f^{-1}(x)$ is countable for all $x\ne
		x_0$.
	\end{description}
	
\end{theorem}
\begin{proof}
	Suppose that $f=g^2$ for some $g\in \mathcal{F}(X)$.
	Consider the action of $g$ on various subsets of $X$ around $x_0$:
	\begin{eqnarray*}
		A_{-2}  \xrightarrow{g} B_{-2}\xrightarrow{g} A_{-1} \xrightarrow{g}B_{-1} \xrightarrow{g} \{x_0\} \xrightarrow{g} \{y_0\},
	\end{eqnarray*}
	where $y_0=g(x_0)$,    $A_{-1} =f^{-1}(x_0)$, $A_{-2}  =f^{-2}(x_0)$, $B_{-1}=g(A_{-1})$ and $B_{-2}=g(A_{-2})$. Let $\tilde{A}_{-1}=f(A_{-2})$ and
	$\tilde{B}_{-1}=g(\tilde{A}_{-1})$.
	Then $\tilde{A}_{-1}\subseteq A_{-1}$, 	$\tilde{B}_{-1}\subseteq B_{-1}\subseteq g^{-1}(x_0)$ and $B_{-2}\subseteq
	g^{-1}(A_{-1})$. 
	Also, since $f(x_0)\ne x_0$, we have $y_0\neq x_0$, $x_0 \notin A_{-1}$ and $x_0\notin B_{-1}$.  
	
	\noindent {\bf Case (i):}  	Since $x_0\notin A_{-1}$ and $\tilde{A}_{-1}\subseteq A_{-1}$, we have $\#f^{-1}(x)\leq 1$ for all $x\in \tilde{A}_{-1}$. Therefore, as $A_{-2}\subseteq \bigcup_{x\in \tilde{A}_{-1}}f^{-1}(x)$ and $\#A_{-2}>1$, it follows that $\#\tilde{A}_{-1}>1$. 
	
	On the other hand, since $y_0\ne x_0$, we have $\#f^{-1}(y_0)\le 1$. This implies that $\#\tilde{B}_{-1}=1$, because $\tilde{B}_{-1}\ne \emptyset$ and $\tilde{B}_{-1}\subseteq B_{-1}\subseteq f^{-1}(y_0)$.
	Therefore, as $x_0\notin \tilde{B}_{-1}$, we get that $\#f^{-1}(\tilde{B}_{-1})=1$. Consequently, $\#B_{-2}=1$, implying that $\#\tilde{A}_{-1}=\#f(A_{-2})=
	\#g(B_{-2})=1$. 	This contradicts an earlier
	conclusion. Hence   $f$ has no square roots in $\mathcal{F}(X)$.
	
	\noindent  {\bf Case (ii):} Since $x_0\notin A_{-1}$ and $\tilde{A}_{-1}\subseteq A_{-1}$, we see that $f^{-1}(x)$ is finite for all $x\in \tilde{A}_{-1}$. Therefore, as $A_{-2}\subseteq \bigcup_{x\in \tilde{A}_{-1}}f^{-1}(x)$ and $A_{-2}$ is infinite, it follows that $\tilde{A}_{-1}$ is infinite. 
	
	On the other hand, 
	since $y_0\ne x_0$, we have that $f^{-1}(y_0)$ is finite. This implies that $\tilde{B}_{-1}$ is finite, because $\tilde{B}_{-1} \subseteq B_{-1}\subseteq f^{-1}(y_0)$. Also, 
	as $x_0\notin B_{-1}$ and $\tilde{B}_{-1} \subseteq B_{-1}$, we see that $f^{-1}(x)$ is finite for all $x\in \tilde{B}_{-1}$.
	Then it follows that $f^{-1}(\tilde{B}_{-1})$ is finite. Consequently, $B_{-2}$ is finite, implying that $\tilde{A}_{-1}=f(A_{-2})=
	g(B_{-2})$ is finite. 	This contradicts the conclusion of the previous paragraph.
	Hence   $f$ has no square roots in $\mathcal{F}(X)$.

	\noindent {\bf Case (iii):} The proof of  Case (ii) repeatedly uses the  fact that a
	finite union of finite sets is finite. The proof for this case is similar, using the result that a countable union of
	countable sets is countable.
\end{proof}

It is worth noting that $f^{-2}$ in the previous theorem cannot be
replaced by $f^{-1}$, as seen from the following.
\begin{example}
	Consider the continuous map $f:[0,1]\to [0,1]$ defined by
	\begin{eqnarray*}
		f(x)= \left\{ \begin{array}{cll}
			\frac{3}{4}& \text{if}&0\leq x\leq \frac{1}{2},\\
			\frac{5}{8}+\frac{x}{4}& \text{if}& \frac{1}{2}< x\leq 1.\end{array}\right.
	\end{eqnarray*}
	Then $f(\frac{3}{4})\neq \frac{3}{4}$, $f^{-1}(\frac{3}{4})$ is
	uncountable and $f^{-1}(x)$ is finite for all $x\neq \frac{3}{4}.$
	However, $f$ has a square root
	\begin{eqnarray*}
		g(x)= \left\{ \begin{array}{cll}
			1& \text{if}&0\leq x\leq \frac{1}{2},\\
			\frac{5}{4}-\frac{x}{2}&\text{if}& \frac{1}{2}< x\leq 1\end{array}\right.
	\end{eqnarray*}
	on $[0,1]$ that is even continuous.
\end{example}

The   above results are given for an arbitrary set $X$. In the rest
of the section and those that follow, we consider $X$ to be a
topological space or a specific metric space and study continuous
roots of continuous self-maps on $X$.  It is useful to
have some notations in this context. Recall that for a topological space $X$, the
space of all continuous self-maps of $X$ is denoted by $\mathcal
{C}(X)$. For each $A\subseteq X$, let $A^0$ denote the interior of
$A$, $\overline{A}$ the closure of $A$, and $\partial A$ the
boundary of $A$.  In particular, if $X$ is a metric space equipped with
metric $d$, then for each $x\in X$ and $\epsilon
>0$, let $B_{\epsilon }(x):=\{y\in X: d(y,x)<\epsilon \}$, the open 
ball in $X$ around $x$ of radius $\epsilon$. For any index set $\Lambda$,
let $S_\Lambda$  denote the group of all permutations (bijections)  on
$\Lambda$, and for convenience, we denote $S_\Lambda$ by $S_k$ when 
$\Lambda=\{1,2,\ldots, k\}$.

We motivate the next theorem through two examples. The
map $f_1(x)=1-x$ on the unit interval $[0,1]$ as an element of ${\mathcal
	C}([0,1])$, with $[0,1]$ in the usual metric induced by $|\cdot|$, has no continuous square roots on $[0,1]$ (see \cite[pp.425-426]{kuczma1990}).
On the other hand,  the map $f_2(x,y) = (1-x, 1-y)$ on $[0,1]\times [0,1]$ as an element of ${\mathcal C}([0,1]\times [0,1])$, where $[0,1]\times [0,1]$ has the metric induced by the norm  $\|(x,y)\|_\infty:=\max\{|x|,|y|\}$, has
a continuous square root  $g(x,y)= (y,1-x)$ on $[0,1]\times [0,1]$.
The reason for these contrasting conclusions can be detected by observing the fixed points of
these maps.

Let $f$ be a continuous self-map on a topological space $X$ and $Y$ a non-empty subset of $X$ invariant under
$f$, i.e., $f(Y)\subseteq Y$. Then it is clear that if $x$ and $x'$
are path-connected in $Y$, then so are $f(x)$ and $f(x')$. Therefore, if $Y=\bigcup _{\alpha\in \Lambda}Y_\alpha$
is the decomposition of $Y$ into its path components for some index set $\Lambda$, then there exists a unique map $\sigma _{f,Y}:
\Lambda \to \Lambda$ such that  $f(x)\in Y_\beta$ whenever $x\in Y_\alpha$
and $\sigma _{f,Y}(\alpha)=\beta$ (Note that $y$ and $y'$  are path-connected in $Y$ if and only if there exists an $\alpha\in \Lambda$ such that
$y,y'\in Y_\alpha$). We call $\sigma _{f,Y}$ as the
map {\em induced} by the map $f$ and the invariant set
$Y$.
Let $F(f):=\{x\in X: f(x)=x\}$, the set of all fixed points of
$f$ in $X$, and $E(f):= R(f)\setminus F(f)$, the complement of $F(f)$ in $R(f)$.


\begin{theorem}\label{connected components}
	Let $n\in \mathbb{N}$ and $f\in {\mathcal C}(X)$ be such that $E(f)$ is invariant under $f$.  
	If $f=g^n$
	for some $g\in {\mathcal C}(X)$, then $E(f)$ is invariant under $g$ and the induced maps satisfy that
	$\sigma _{f,E(f)} = \sigma _{g, E(f)}^n$. In other words, if $\sigma _{f,
		E(f)}$ has no $n^{\rm th}$ roots, then $f$ also has no
	continuous $n^{\rm th}$ roots on $X$.
\end{theorem}

\begin{proof}
	Let $E(f)=\bigcup _{\alpha\in \Lambda}Y_\alpha$
	be the decomposition of $E(f)$ into its path components for some index set $\Lambda$.
	The case $n=1$ is trivial. So, let $n>1$, and consider any $x\in E(f)$. Then $x=f(x')$ for some $x'\in X$. Clearly, $g(x)\in R(f)$, because $g(x)= g(f(x'))= g(g^n(x'))=g^n(g(x'))= f(g(x'))$. Also, if $g(x)\in F(f)$, then
	$g(x)=f(g(x))= g^n(g(x))=g(f(x))$, implying that  $f(x)=g^n(x)=g^n(f(x))=f^2(x)$, i.e., $f(x)$ is a fixed point of
	$f$. This contradicts the hypothesis that $E(f)$ is invariant under $f$. Therefore $g(x)\in E(f)$.
	
	%
	The equality $\sigma _{f,Y}= \sigma _{g, Y}^n$ follows from the
	definitions  of $\sigma_{f,E(f)}$ and $\sigma_{g,E(f)}$, and the relation $f=g^n.$
\end{proof}
In the remainder of this section, we illustrate Theorem \ref{connected components} through some simple
examples.  Remark that $\sigma \in S_k$, the group
of permutation of $k$ elements,  has a square root if and only if
the number of even cycles of $\sigma$ is even. More generally,
$\sigma \in S_k$ has a $n^{\rm th}$ root if and only if for every $m
= 1,2, \dots$ it is true that the number of $m$-cycles that $\sigma$
has is a multiple of $((m,n))$, where
$((m,n)):=\prod_{p|m}^{}p^{e(p,n)}$, $e(p,n)$ being the highest
power of $p$ dividing $n$ (see \cite[pp.155-158]{wilf2006}).


\begin{corollary}\label{C1}
	Let $a<b$ be real numbers and $f\in \mathcal{C}([a,b])$ be such
	that $f|_{R(f)}$ is a non-constant strictly decreasing map. Then $f$
	has no iterative roots of even orders in $\mathcal{C}([a,b])$.
\end{corollary}

\begin{proof}
	Since $f:[a,b]\to [a,b]$ is a non-constant strictly decreasing continuous map, we see that $R(f)$ is a compact interval $[c,d]$ in $\mathbb{R}$ for some $c<d$ in $[a,b]$ and $f$ has a unique
	fixed point $p$ in $(c, d)$.     Also, $E(f)=[c,p)\cup (p, d]$ is the decomposition of $E(f)$ into path
	components and  $\sigma :=\sigma _{f, E(f)}$ is the transposition of the two point set $S_2$.  Since $\sigma$ has no iterative roots of even orders in $S_2$, it follows from Theorem \ref{connected components}
	that $f$ has no iterative roots of even orders in $\mathcal{C}([a,b])$.
\end{proof}

\begin{example}
	Consider $[0,1]\times [0,1]$ in the topology induced by the norm $\|(x,y)\|_\infty:=\max\{|x|,|y|\}$, and let $f_1, f_2$ be the continuous maps on $[0,1]\times [0,1]$
	defined by $f_1(x,y)=(1-x, 0)$ and $f_2(x,y)=(y,x)$. Then Theorem
	\ref{connected components} is applicable to $f_1$ and $f_2$ with
	\begin{eqnarray*}E(f_1)&=& \left\{(x,0): 0\le x<\frac{1}{2}\right\}\cup
		\left\{(x,0):\frac{1}{2}<x\le 1\right\};\\
		E(f_2)&=& \left\{ (x,y): 0\leq x<y\leq 1\right\}\cup \left\{(x,y): 0\leq y<x\leq
		1\right\}.\end{eqnarray*}
	It follows that both $f_1$ and $f_2$ have no iterative roots of even orders in
	$\mathcal{C}([0,1]\times [0,1])$.
\end{example}

%


\section{Continuous functions on the unit cube of ${\mathbb R}^m$}\label{sec3}

As seen in the Introduction, plentiful  literature is available on iterative square roots of functions on compact real intervals. In this section, we extend some of the results in \cite{Humke-Laczkovich} to the {\it unit cube}
\begin{eqnarray*}
	I^m :=\{ (x_1, x_2, \ldots , x_m)\in {\mathbb R}^m: x_i\in I~\text{for}~ 1\leq i\leq m\}
\end{eqnarray*}
considered in the metric induced by the norm
\begin{eqnarray}\label{norm}
\|(x_1, x_2,\ldots,x_m)\|_\infty=\max\{|x_i|: 1\le i\le m\},
\end{eqnarray}
where $I:=[0,1]$.
In studying iterative roots of continuous maps on intervals,
one observes that it becomes essential to look at piecewise linear
maps. Here we study their analogues in higher dimensions. 
It is convenient to have the following elementary notions and notations from the theory of simplicial complexes to present our proofs.

As defined in \cite{deo}, a set $S=\{x_0, x_1, \ldots, x_k\}$ in $\mathbb{R}^m$, where $k\ge 1$, is said to be {\it geometrically independent} if the set $\{x_1-x_0, x_2-x_0, \dots, x_k-x_0\}$ is linearly independent  in $\mathbb{R}^m$. Equivalently, $S$ is geometrically independent if and only if for arbitrary reals $\alpha_i$, the equalities $\sum_{i=0}^{k}\alpha_ix_i=0$ and $\sum_{i=0}^{k}\alpha_i=0$ imply that $\alpha_i=0$ for all $0\le i\le k$.
%
A set having only one point is assumed to be geometrically independent. If $S$ is geometrically independent, then  there exists a unique $k$-dimensional hyperplane which passes through all the points of $S$ and each point $x$ on this hyperplane
can be expressed uniquely as
$x=\sum_{i=0}^{k}\alpha_ix_i$ such that $\sum_{i=0}^{k}\alpha_i=1$. The real numbers $\alpha_0, \alpha_1, \ldots, \alpha_k$, which are uniquely determined by $S$, are
called the {\it barycentric coordinates} of the point $x$ with respect to the set $S$.

Let $S=\{x_0, x_1, \ldots, x_k\}$ be a geometrically independent set
in $\mathbb{R}^m$, where $0\leq k\leq m$. Then the {\it $k$-dimensional
	geometric simplex} or {\it $k$-simplex} spanned by $S$, denoted by
$\sigma$, is defined as the set of all points $x\in \mathbb{R}^m$
such that $x=\sum_{i=0}^{k}\alpha_ix_i$, where
$\sum_{i=0}^{k}\alpha_i=1$ and $\alpha_i\ge 0$ for all $0\le i \le k$. Equivalently, $\sigma$ is the convex hull of $S$.
The points $x_0, x_1, \ldots, x_k$ are called the {\it vertices} of $\sigma$, and we usually write $\sigma=\langle x_0, x_1, \ldots, x_k\rangle$ to indicate that $\sigma$ is the $k$-simplex with vertices $x_0, x_1, \ldots, x_k$.
If $\sigma=\langle x_0, x_1, \ldots, x_k\rangle$ is a $k$-simplex, then the set of
those points of $\sigma$ for which all barycentric coordinates are strictly positive, is
called the {\it open $k$-simplex} $\sigma$ (or the {\it interior} of $\sigma$), and is denoted by $\sigma^0$.
If $\sigma$ is a $p$-simplex and $\tau$ is a $q$-simplex in $\mathbb{R}^m$ such that $p\le q\le m$, then we say that $\sigma$ is a {\it $p$-dimensional face} (or simply a {\it $p$-simplex}) {\it of $\tau$} if each vertex of $\sigma$ is also a vertex of $\tau$.  If $\sigma$ is a face of $\tau$ with $p<q$, then $\sigma$ is called a {\it proper face of $\tau$}.  Any $0$-dimensional face of a simplex is simply a vertex of the simplex and  a $1$-dimensional
face of a simplex is usually called an {\it edge} of that simplex.

A finite (resp. countable) {\it simplicial complex} $\mathcal{K}$ is a finite (resp. countable) collection of simplices of $\mathbb{R}^m$ satisfying the following
conditions: {\bf (i)} If $\sigma \in \mathcal{K}$, then all the faces of $\sigma$ are in $\mathcal{K}$; {\bf (ii)} If $\sigma, \tau \in \mathcal{K}$, then either $\sigma \cap \tau =\emptyset$ or $\sigma\cap \tau$ is a common face of both $\sigma$ and $\tau$.
Let $\mathcal{K}$ be a finite simplicial complex and
$|\mathcal{K}|=\bigcup_{\sigma \in \mathcal{K}}\sigma$ be the union
of all simplices of $\mathcal{K}$. Then $|\mathcal{K}|$ is a
topological space with the topology inherited from $\mathbb{R}^m$.
The space $|\mathcal{K}|$ is called the {\it geometric carrier} of
$\mathcal{K}$. A topological subspace of $\mathbb{R}^m$, which is
the geometric carrier of some finite simplicial complex, is called a
{\it rectilinear polyhedron}. A topological space $X$ is said to be
a {\it polyhedron} if there exists a finite simplicial complex
$\mathcal{K}$ such that $|\mathcal{K}|$ is homeomorphic to $X$. In
this case, the space $X$ is said to be {\it triangulable} and
$\mathcal{K}$ is called a {\it triangulation} of $X$. A space $X$ is
said to be {\it countably triangulable} if there exists a countable
triangulation of $X$, i.e.,  there exists a countable simplicial
complex $\mathcal{K}$ of simplices of $\mathbb{R}^m$  such that its
geometric carrier $|\mathcal{K}|$ is homeomorphic to $X$. Further,
we do not distinguish  between $X$ and its geometric carrier
$|\mathcal{K}|$ whenever $X=|\mathcal{K}|$ for some (finite or
countable) triangulation $\mathcal{K}$ of $X$, because in that case
we always consider the identity map {\rm id} to be the homeomorphism
between $|\mathcal{K}|$ and $X$.

If $\sigma=\langle x_0, x_1, \ldots, x_k\rangle$ is a $k$-simplex in $\mathbb{R}^m$, then the point $\sum_{i=0}^{k}\frac{1}{k+1}x_i$ is called the {\it barycentre} of $\sigma$ and is denoted by $\dot{\sigma}$. In other words, barycentre of $\sigma$ is that point of $\sigma$ whose barycentric coordinates, with respect to each of the vertices of $\sigma$, are equal. Given a simplicial complex $\mathcal{K}$, let $\mathcal{K}^{(1)}$ be a simplicial complex whose vertices are barycentres of all simplices of $\mathcal{K}$, and for any distinct simplices $\sigma_1, \sigma_2, \ldots, \sigma_k$ of $\mathcal{K}$, $\langle \dot{\sigma_1}, \dot{\sigma_2}, \ldots, \dot{\sigma_k}\rangle$ is a simplex of $\mathcal{K}^{(1)}$ if and only if $\sigma_i$ is a face of $\sigma_{i+1}$ for $i=0, 1,\ldots, k-1$. Then $\mathcal{K}^{(1)}$ is a simplicial complex and is called the {\it first barycentric subdivision} of $\mathcal{K}$. By induction, we define $l^{\rm th}$ barycentric subdivision $\mathcal{K}^{(l)}$ of $\mathcal{K}$ to be the first barycentric subdivision of $\mathcal{K}^{(l-1)}$ for each $l>1$. We also put $\mathcal{K}^{(0)}=\mathcal{K}$ for convenience. For a simplicial complex $\mathcal{K}$, we define the {\it mesh} of $\mathcal{K}$, denoted by $mesh(\mathcal{K})$, as
\begin{eqnarray*}
	mesh(\mathcal{K})=\max\{diam(\sigma): \sigma ~\text{is a simplex of}~\mathcal{K}\},
\end{eqnarray*}
where $diam(\sigma)$ denote the diameter of $\sigma$. Then $\lim\limits_{l\to \infty}mesh(\mathcal{K}^{(l)})=0$ for every non-empty simplicial complex $\mathcal{K}$.

An $f\in \mathcal{F}(\mathbb{R}^m)$ is said to  be {\it affine
	linear} if $f(\cdot)=T(\cdot)+z$ for some linear map $T\in
\mathcal{F}(\mathbb{R}^m)$ and  $z\in \mathbb{R}^m$. Clearly, each
affine linear map on $\mathbb{R}^m$ is continuous. An affine linear
map $f$ is said to be of {\it full rank} if the associated linear
map $T$ is invertible. Note that full rank affine
linear maps are open maps, i.e., they map open sets to open sets.
They also map geometrically independent sets to geometrically
independent sets. We sew affine linear maps on simplices piece by
piece to get many continuous maps on $\mathbb{R}^m$.

An ${\mathbb R}^m$-valued function $f$ on a subset $X$ of ${\mathbb
	R}^m$ is said to be {\it piecewise affine linear} (resp. {\it
	countably piecewise affine linear}) if $X=|\mathcal{K}|$ for some
finite (resp. countable) simplicial complex $\mathcal{K}$ of
simplices of $\mathbb{R}^m$  and $f|_\sigma$ is the restriction  of some affine linear map $\phi_\sigma$ on $\mathbb{R}^m$ to
$\sigma $
for every $\sigma \in \mathcal{K}$. In this case we say that $f$ is
supported by the simplicial complex ${\mathcal K}.$

Here is a basic construction of piecewise affine linear maps that
we use repetitively.  Let $\mathcal{K}$ be a simplicial
complex in $\mathbb{R}^m$ and $U=\{x_0, x_1, \ldots \}$
be its set of vertices ordered in some arbitrary way. Let $V=\{y_0, y_1, \ldots\}$ be an ordered set of points in
${\mathbb R}^m$. Then we can define a  ${\mathbb R}^m$-valued function $f_{U,V}$ on $|\mathcal{K}|$  satisfying
$f(x_j)=y_j$ for all $j\in \mathbb{Z}_+$
as follows: If $x\in |\mathcal{K}|$, then $x$ is in the interior of an
unique $k$-simplex $\langle x_{i_0}, x_{i_1}, \dots, x_{i_k}\rangle$ in $\mathcal{K}$ for some $0\le k\le m$.
We write $x=\sum_{j=0}^{k}\alpha_{i_j}x_{i_j}$ with $\sum_{j=0}^{k}\alpha_{i_j}=1$ and define
\begin{eqnarray}\label{f_0}
f_{U,V}(x)=\sum_{j=0}^{k}\alpha_{i_j}y_{i_j}.
\end{eqnarray}
Clearly, $f_{U,V}$ is a 
continuous piecewise affine linear map on $|\mathcal{K}|$ satisfying that
$f(x_j)=y_j$ for every $j\in \mathbb{Z}_+$. Furthermore, if $|{\mathcal K}|$ is convex and
$y_j\in |\mathcal {K}|$ for every $j$, then
$f_{U,V}$ is a self-map of $|{\mathcal K}|.$
\begin{lemma}\label{L1}
	Let ${\mathcal K}$ and $f_{U,V}$ be as described above. \begin{description}
		\item[{\rm (i)}] If $\sigma=\langle x_{i_0}, x_{i_1}, \dots, x_{i_k}\rangle$ is a $k$-simplex in $\mathcal{K}$ and
		$\{y_{i_0}, y_{i_1}, \dots, y_{i_k}\}$  is geometrically independent in ${\mathbb R}^m$, then  $f_{U,V}|_\sigma$ is injective and $f(\sigma)^0=f(\sigma^0)$.
		\item[{\rm (ii)}] If  $V=\{y_0, y_1, \ldots\}$ and $W=\{z_0, z_1, \ldots\}$ are ordered
		sets in  $|\mathcal{K}|$ such that $\|y_j-z_j\|_\infty\le \eta$ for all $j\in
		\mathbb{Z}_+$, where $\eta>0$, then
		$\|f_{U,V}(x)-f_{U,W}(x)\|_\infty\le \eta$ for all $x\in
		|\mathcal{K}|$.
	\end{description}
\end{lemma}
\begin{proof}
	The result (i) is trivial.  For each $x\in |\mathcal{K}|$, by \eqref{f_0} we have
	\begin{eqnarray*}
		\|f_{U,V}(x)-f_{U,W}(x)\|_\infty\le \sum_{j=0}^{k}\alpha_{i_j}\|y_{i_j}-z_{i_j}\|_\infty\le \eta,
	\end{eqnarray*}
	proving result  (ii).
\end{proof}

\begin{lemma}
	If  $f\in \mathcal{F}(\mathbb{R}^m)$ is an affine linear map such that $f|_S=0$ for a geometrically independent set $S=\{x_0, x_1, \ldots, x_m\}$ of $\mathbb{R}^m$, then $f=0$ on $\mathbb{R}^m$.
\end{lemma}
\begin{proof}
	Let $f(\cdot )=T(\cdot ) +z$ for some $z\in \mathbb{R}^m$ and a
	linear map $T\in \mathcal{F}(\mathbb{R}^m)$. Since $S$ is
	geometrically independent, we have $B=\{x_1-x_0, x_2-x_0, \dots,
	x_m-x_0\}$ is linearly independent, and therefore it is a Hamel
	basis for $\mathbb{R}^m$. To prove $f= 0$, consider any $x\in
	\mathbb{R}^m$. Then $x=\sum_{j=1}^{m}\alpha_j(x_j-x_0)$ for some
	$\alpha_1, \alpha_2, \ldots, \alpha_m\in \mathbb{R}$, implying that
	\begin{eqnarray*}
		f(x)&=&T\left(\sum_{j=1}^{m}\alpha_j(x_j-x_0)\right)+z\\
		&=&\sum_{j=1}^{m}\alpha_jT(x_j)-\sum_{j=1}^{m}\alpha_jT(x_0)+z\\
		&=&\sum_{j=1}^{m}\alpha_j(-z)-\sum_{j=1}^{m}\alpha_j(-z)+z=z.
	\end{eqnarray*}
	In particular, $f(x_0)=z$, and therefore we get that  $z=0$. Hence $f=0$ on $\mathbb{R}^m$.
\end{proof}

Using the above lemma, we can  deduce the following result on uniqueness of affine linear maps.

\begin{corollary}\label{C2}
	If $f_1, f_2 \in \mathcal{F}(\mathbb{R}^m)$ are affine linear maps such that $f_1|_\sigma= f_2|_\sigma$ for
	some $m$-simplex $\sigma\in \mathcal{K}$, then $f_1=f_2$ on
	$\mathbb{R}^m$. Consequently, if $f_1|_U=f_2|_U$ for some open
	subset $U$ of $\mathbb {R}^m$, then $f_1=f_2$ on $\mathbb{R}^m.$
\end{corollary}
\begin{proof}
	Follows from the above lemma, because $f_1-f_2$ is an affine linear map on $\mathbb{R}^m$ that vanishes on the geometrically independent set of all vertices of $\sigma$.
\end{proof}

\begin{corollary}\label{C3}
	Let $f \in \mathcal{F}(\mathbb{R}^m)$ be an affine linear map and
	$\sigma$ be an $m$-simplex in $\mathbb{R}^m$ with vertices $x_0, x_1,
	\ldots, x_m$ such that $y_0, y_1, \ldots, y_m$ are geometrically
	independent  in $\mathbb{R}^m$, where $y_j=f(x_j)$ for all
	$0\leq j\leq m$.   Then the following statements are
	true.\begin{description}
		\item[{\rm (i)}]  $f(U)$ is open for every open set $U$ in
		the interior of $\sigma$.
		\item[{\rm (ii)}] $f$ maps geometrically independent subsets of
		${\mathbb R}^m$ to geometrically independent subsets.
		\item[{\rm (iii)}] If $y_j\neq x_j$ for some $0\le j\le m$ and $\{v_0, v_1, \ldots , v_m\}$ is a geometrically independent subset of
		$\sigma$, then there exists at least one $j$ such that $f(v_j)\neq v_j$.
	\end{description}
\end{corollary}
\begin{proof}
	Results   (i) and  (ii) are trivially true, because geometric independence of
	$\{y_0, y_1, \ldots, y_m\}$ implies that $f$ is of full rank. If
	$f(v_j)=v_j$ for all $0\le j\le m$, then  $f={\rm id}$ on $\mathbb{R}^m$, implying that $x_j=f(x_j)=y_j$ for all $0\le
	j\le m$. Therefore (iii) follows.
\end{proof}

For each subset $S$ of $\mathbb{R}^m$, let
\begin{eqnarray*}
	{\rm Aff}(S):=\left\{\sum_{j=1}^{l}\alpha_jx_j: x_j\in S~\text{for}~1\le j\le l~\text{and}~\sum_{j=1}^{l}\alpha_j=1\right\}
\end{eqnarray*}
and
\begin{eqnarray*}
	{\rm Aff}_0(S):=\left\{\sum_{j=1}^{l}\alpha_jx_j: x_j\in S~\text{for}~1\le j\le l~\text{and}~\sum_{j=1}^{l}\alpha_j=0\right\}.
\end{eqnarray*}
Then ${\rm Aff}_0(S)$ is a vector subspace of $\mathbb{R}^m$ and
${\rm Aff}(S)={\rm Aff}_0(S)+x$ for some $x\in S$, where ${\rm
	Aff}_0(S)+x:=\{y+x:y\in {\rm Aff}_0(S)\}$.
We say that ${\rm Aff}(S)$ has dimension $k$, written as $dim({\rm Aff}(S))=k$, if
${\rm Aff}_0(S)$ is a $k$-dimensional vector subspace of $\mathbb{R}^m$. A point $z\in \mathbb{R}^m$ is said to be
{\it geometrically independent of $S$} if $z\notin {\rm Aff}(S)$. Then  it
is clear that $z$ is geometrically independent of $S$ if and
only if $0$ is geometrically independent of $S-z$.

\begin{lemma}\label{L3}
	Let $S$ be a countable subset of $\mathbb{R}^m$. Then
	for each $z\in \mathbb{R}^m$ and $\zeta>0$, there exists a $y\in B_\zeta(z)$
	such that $y$ is geometrically independent of $Z$ for all subset $Z$ of $S$ such that $\#Z\le m$.
\end{lemma}
\begin{proof}
	Since $dim({\rm Aff}(Z))\le m-1$, clearly ${\rm Aff}(Z)\bigcap \overline{(B_{\frac{\zeta}{2}}(z))}$
	is a nowhere dense subset of the complete
	metric space $\overline{(B_{\frac{\zeta}{2}}(z))}$ for
	each subset $Z$ of $S$ such that $\#Z\le m$. Therefore by Baire category theorem, we have
	\begin{eqnarray*}
		\overline {(B_{\frac{\zeta}{2}}(z))}\ne \bigcup\left({\rm Aff}(Z)\bigcap \overline {(B_{\frac{\zeta}{2}}(z))}\right),
	\end{eqnarray*}
	where the union is taken over all subsets $Z$ of $S$ such that
	$\#Z\le m$. Hence there exists a $y\in \overline{
		(B_{\frac{\zeta}{2}}(z))}$ such that $y$ is geometrically
	independent of $Z$ for all subset $Z$ of $S$ such that $\#Z\le m$. Since $\overline{
		(B_{\frac{\zeta}{2}}(z))}\subseteq B_\zeta(z)$, the result follows. 
\end{proof}

\begin{lemma}\label{L5}
	Let $z_1, z_2, \ldots$ be
	a sequence of points in $\mathbb{R}^m$ and $\zeta_1, \zeta_2,
	\ldots$ be a sequence of positive reals. Then there exists a
	sequence $y_1, y_2, \ldots$ in $\mathbb{R}^m$ such that $y_j\in
	B_{\zeta_j}(z_j)$ for each $j\in \mathbb{N}$ and  every subset $Z$
	of $S=\{y_1, y_2, \ldots\}$ with $\#Z\le m+1$ is geometrically
	independent.
\end{lemma}
\begin{proof}
	Given $z_1\in \mathbb{R}^m$ and $\zeta_1>0$, consider any $y_1\in
	B_{\zeta_1}(z_1)$ such that $y_1\ne z_1$. Then clearly any subset $Z$ of
	$S_1=\{y_1\}$ with $\#Z\le m+1$ is geometrically independent. Next,
	by induction, suppose that $k>1$, and there exist $y_1, y_2, \ldots,
	y_k \in \mathbb{R}^m$ with $y_j\in B_{\zeta_j}(z_j)$ for $1\le j\le
	k$ such that $Z$ is geometrically independent for all subsets $Z$ of
	$S_k=\{y_1, y_2, \ldots, y_k\}$ with $\#Z\le m+1$. Then by Lemma
	\ref{L3} there exists $y_{k+1}\in B_{\zeta_{k+1}}(z_{k+1})$ such
	that $y_{k+1}$ is geometrically independent of $Z$ for all subset
	$Z$ of $S_k$ such that $\#Z\le m$. This implies that $Z$ is
	geometrically independent for all subset $Z$ of $S_{k+1}$ such that
	$\#Z\le m+1$, proving the result for $k+1$. Thus, continuing the
	above process we get a sequence $y_1, y_2, \ldots$ in $\mathbb{R}^m$
	such that $y_j\in B_{\zeta_j}(z_j)$ for each $j\in \mathbb{N}$
	satisfying the desired property.
\end{proof}


In addition to the above notions and results on simplicial complexes, we  need the following lemma to prove our results.

\begin{lemma}\label{L2}
	Let $X$ be a metric space with metric $d$ and $f\in \mathcal{C}(X)$. If $x\in X$ is not a fixed point of $f$, then there exists an open set $V$ in $X$ containing $x$ such that  $f^{-1}(V)\cap V=\emptyset$ and $f(V)\cap V=\emptyset$.
\end{lemma}
\begin{proof}
	Let $\epsilon=d(x,f(x))$. Since $f(x)\ne x$,  clearly  $\epsilon>0$. Since $f$ is continuous at $x$, there exists $\delta>0$ with $\delta<\frac{\epsilon}{4}$ such that
	\begin{eqnarray}\label{6}
	d(f(x), f(y)) <\frac{\epsilon}{4}~~\text{whenever}~~ y\in X~~\text{with}~~d(x,y)<\delta.
	\end{eqnarray}
	Let $V:=B_\frac{\delta}{2}(x)$.
	
	\noindent {\bf Claim:} $f^{-1}(V)\cap V=\emptyset$  and $f(V)\cap
	V=\emptyset$.
	
	Suppose that $y\in f^{-1}(V)\cap V$. Then $f(y)\in V$, implying that
	$d(y, f(y))\le d(y,x)+d(x,f(y))<\delta$. Also,  by \eqref{6} we have
	$d(f(x), f(y))<\frac{\epsilon}{4}$. Therefore $\epsilon=d(x,
	f(x))\le d(x, y)+d(y, f(y))+d(f(y), f(x))<\frac{3\epsilon}{4}$ is a
	contradiction. Next, suppose that $y\in f(V)\cap V$. Then $y=f(z)$
	for some $z\in V$. Therefore, by \eqref{6} we have $d(f(x),
	y)<\frac{\epsilon}{4}$, implying that $\epsilon=d(x, f(x))\le d(x,
	y)+d(y, f(x))<\frac{\epsilon}{2}$, which is a contradiction. Thus the claim
	holds, and the result follows.
\end{proof}

Having the above machinery, we are ready to prove our desired results.
\begin{theorem}\label{Thm1}
	If $h\in \mathcal{C}(I^m)$, then for each $\epsilon
	>0$ the ball $B_{\epsilon }(h)$ of radius $\epsilon$ around $h$ contains a map $f$
	having no (even discontinuous) iterative square roots on $I^m$. In particular,
	$\mathcal{W}(2;I^m)$
	does not contain any ball of $\mathcal{C}(I^m)$.
\end{theorem}

\begin{proof}
	Let $h\in \mathcal{C}(I^m)$  and $\epsilon>0$ be arbitrary.
	Our strategy to prove the existence of an  $f\in B_\epsilon(h)$ with no square roots is to make use of
	Theorem \ref{T3}, and we construct this $f$ in a few steps.
	
	\noindent {\it Step 1: Construct a
		piecewise affine linear map $f_0\in B_\epsilon(h)$ supported by some suitable triangulation $\mathcal {K}$ of
		$I^m$.}

	Since $h$ is uniformly continuous on $I^m$, there exists a $\delta>0$ such that
	\begin{eqnarray}\label{1}
	\|h(x)-h(y)\|_\infty <\frac{\epsilon }{10} ~~\text{whenever}~~x, y\in I^m~~\text{with}~~\|x-y\|_\infty<\delta.
	\end{eqnarray}
	Being a polyhedron, $I^m$ is  triangulable.
	Consider a  triangulation $\mathcal{K}$  of $I^m$ with vertices
	$x_0, x_1, x_2, \ldots, x_r$ such that $\|x_{i_0}-x_{i_1}\|_\infty<\frac{\delta}{4}$ for all $2$-simplex $\langle x_{i_0},x_{i_1}\rangle$ of $\sigma$ and for all $m$-simplex $\sigma \in \mathcal{K}$. Choose
	$y_0, y_1, y_2,\ldots, y_r \in I^m$ inductively such that $y_j\ne x_j$ and
	\begin{eqnarray}\label{2}
	\|h(x_j)-y_j\|_\infty <\frac{\epsilon }{20}~~\text{for all}~~0\le j\le r,
	\end{eqnarray}
	and $\{y_{i_0}, y_{i_1}, \ldots, y_{i_m}\}$ is geometrically
	independent whenever $\langle x_{i_0}, x_{i_1},\dots,
	x_{i_m}\rangle \in \mathcal{K}$ for $0\le i_0, i_1, \ldots, i_m\le
	r$. This is possible by Lemma \ref{L5}. Here, while choosing $y_j$'s with $y_j\in B_{\frac {\epsilon}{20} }(h(x_j))$ for $0\le j\le r$ satisfying the geometric independence condition, we have some additional restrictions to fulfill,  viz. $y_j\neq x_j$ and $y_j\in I^m$ for all $0\le j\le r$.  However, the Baire category theorem  provides this flexibility as there are plenty of vectors to choose from.     Then
	\begin{eqnarray}
	\|y_{i_0}-y_{i_1}\|_\infty&\le& \|y_{i_0}-h(x_{i_0})\|_\infty+\|h(x_{i_0})-h(x_{i_1})\|_\infty+
	\|h(x_{i_1})-y_{i_1}\|_\infty\nonumber \\
	&<&  \frac{\epsilon}{20}+\frac{\epsilon }{10}+\frac{\epsilon }{20}=\frac{\epsilon }{5}  \label{5} 
	\end{eqnarray}
	whenever  $\langle x_{i_0},x_{i_1}\rangle$ is a $2$-simplex of $\sigma$ for all $m$-simplex $\sigma \in \mathcal{K}$.
	Let $f_0:|\mathcal{K}|=I^m \to I^m$ be the map $f_{U, V}$ defined as in \eqref{f_0}, where $U$ and $V$
	are the ordered sets $\{x_0, x_1, \ldots, x_r\}$ and  $\{y_0, y_1, \ldots, y_r\}$, respectively.
	Then $f_0$ is a continuous piecewise affine linear self-map of  $I^m$ satisfying that $f_0(x_j)=y_j$
	for all $0\le j\le r$. Also, by using  result  (i) of Lemma
	\ref{L1}, we have $R(f_0)^0\ne \emptyset$ and $f_0|_\sigma$ is
	injective for all $\sigma \in \mathcal{K}$. Further,
	\begin{eqnarray}\label{7}
	\mbox{diam}(f_0(\sigma ))<\frac{\epsilon }{5}~~\text{for all}~~
	\sigma \in \mathcal{K}.
	\end{eqnarray}
	
	\noindent{\bf Claim:} $\rho(f_0,h)<\frac{\epsilon}{6}$.
	
	Consider an arbitrary $x\in I^m$. Then $x$ is in the interior of some unique $k$-simplex $\langle x_{i_0}, x_{i_1}, \dots, x_{i_k}\rangle$  in $\mathcal{K}$ for some $0\le k\le m$. Let $x=\sum_{j=0}^{k}\alpha_{i_j}x_{i_j}$ with $\sum_{j=0}^{k}\alpha_{i_j}=1$.
	Then by \eqref{1} and \eqref{2}, we have
	\begin{eqnarray}
	\|f_0(x)-h(x)\|_\infty&=&\left\|\sum_{j=0}^{k}\alpha_{i_j}y_{i_j}-\sum_{j=0}^{k}\alpha_{i_j}h(x)\right\|_\infty \nonumber\\
	&\le & \sum_{j=0}^{k}\alpha_{i_j}\left\{\|y_{i_j}-h(x_{i_j})\|_\infty+\|h(x_{i_j})-h(x)\|_\infty\right\}\nonumber\\
	&<& \frac{\epsilon }{20}+\frac{\epsilon }{10}\nonumber\\
	&<&\frac{\epsilon}{6},  \label{4}
	\end{eqnarray}
	since $\|x-x_{i_j}\|_\infty\le \mbox{diam}(\langle x_{i_0}, x_{i_1}, \dots, x_{i_k}\rangle)<\frac{\delta}{4}$
	for $0\le j\le k$. This proves the claim and Step 1 follows.

	\noindent {\it Step 2: Obtain  an $m$-simplex $\Delta \in \mathcal
		{K}  $ containing a non-empty open set $Y\subseteq \Delta ^0\cap
		R(f_0)^0$ such that $f_0(Y)\subseteq \Delta _1^0$ for some
		$m$-simplex $\Delta _1$ in ${\mathcal K}.$}
	
	Consider any $m$-simplex $\sigma =\langle x_{i_0}, x_{i_1},\ldots, x_{i_m}\rangle$ in $\mathcal{K}$. Since $\{y_{i_0}, y_{i_1},\ldots, y_{i_m}\}$ is geometrically
	independent by construction, we have $f_0(U)$ is open in $I^m$ for each open set $U$ in
	$\sigma^0$. Choose a non-empty open subset	$Y_0$ of $f(\sigma )^0$ so small that $Y_0\subseteq \Delta^0$ for some
	$m$-simplex $\Delta \in \mathcal {K}$. Then $f_0(Y_0)$ 
	has non-empty interior. Again, choose a sufficiently small non-empty open subset  $Y$ of $Y_0$
	such that $f_0(Y)\subseteq \Delta _1$ for some
	$m$-simplex $\Delta _1\in {\mathcal K}.$ This completes the proof of
	Step 2.

	
	
	\noindent {\it Step 3: Show that there exists a barycentric subdivision $\mathcal {K}^{(s)}$ of $\mathcal {K}$ for some
		$s \in \mathbb {N}$ with an $m$-simplex  $\sigma_0\in \mathcal{K}^{(s)}$  such that $\sigma_0\subseteq Y$, and $f_0^{-1}(\sigma_0)\cap \sigma=\emptyset$ and
		$f_0(\sigma_0)\cap \sigma=\emptyset$ for all $\sigma\in \mathcal{K}^{(s)}$ with $\sigma\cap\sigma_0\ne
		\emptyset$.}
	
	Since $y_j\neq x_j$ for all $0\le j\le r$, by Corollary \ref{C3} we have
	$f_0(z_0)\ne z_0$ for some $z_0\in Y$. Choose
	an open set $W$ in $I^m$  such that $z_0\in W\subseteq Y$
	and $f_0(x)\ne x$ for all $x\in W$. Then using Lemma \ref{L2}, we get an open set $V$ in $I^m$ with $z_0\in V\subseteq W$
	such that $f_0^{-1}(V)\cap V=\emptyset$ and $f_0(V)\cap V=\emptyset$.  Perform a barycentric
	subdivision of $\mathcal{K}$ of order $s$ so large that $V$ contains an $m$-simplex $\sigma _0\in \mathcal{K}^{(s)}$ and all
	$m$-simplices $\sigma\in \mathcal{K}^{(s)}$ adjacent to it, i.e.,  $\sigma \subseteq V$ for all $m$-simplices $\sigma \in \mathcal{K}^{(s)}$ such that $\sigma \cap \sigma _0\ne \emptyset$. Since $\lim\limits_{l\to \infty}mesh(\mathcal{K}^{(l)})=0$, it is possible to obtain such a subdivision of $\mathcal{K}$. This completes the proof of Step 3.
	
	\noindent {\it Step 4: Modify $f_0$ slightly to get a new piecewise affine
		linear map $f\in B_\epsilon (h)$ that is constant on $\sigma
		_0.$}
	
	
	To retain continuity, when we modify $f_0$ on $\sigma _0$,  we must also modify it on all $m$-simplices in $\mathcal{K}^{(s)}$ adjacent to it.  Let $A:=\bigcup_{j=0}^{p}\sigma_j$, 
	where $\sigma_j$ is precisely an $m$-simplex in
	$\mathcal{K}^{(s)}$ such that $\sigma_j\cap \sigma_0\ne \emptyset$ for all $0\le j\le p$.
	Then, it is clear that $A\subseteq V\subseteq \Delta
	^0$.
	

	Let the vertices of $\sigma _0$ be  $u_0,\ldots, u_{m-1}$ and $u_m$,
	and let
	$v_j:=f_0(u_j)$ for all $0\le j \le m$. Then  $v_j\ne u_j$ for all $0\le j\le m$, because $u_j\in V$ for all $0\le j\le m$ and $f_0(V)\cap V=\emptyset$. Also, since $u_0, u_1, \ldots , u_m$ are
	geometrically independent vectors  in $\Delta$, by result  (ii) of
	Corollary \ref{C3}, we see that $v_0, \ldots, v_m$ are geometrically
	independent. Further,  they are all contained in the
	interior of a single simplex $\Delta _1$ of $\mathcal{K}$ by Step 2. Hence, by
	result  (iii) of Corollary \ref{C3}, it follows that $f_0(v_j)\neq v_j$ for some $0\le j\le m$.
	Without loss of generality, we assume that $f_0(v_0)\neq v_0.$

	Extend $\{u_0, \ldots, u_m\}$ to  $\{u_0, \ldots , u_k\}$ with
	$k> m$  to include the vertices of all the  simplices of ${\mathcal K}^{(s)}$, and let $v_j:=
	f_0(u_j)$ for all $0\leq j\leq k$. Now, define $f$ by
	\begin{eqnarray*}
		f(u_j) = \left\{ \begin{array}{lll}
			v_0&\text{if}& 0\leq j\leq m,\\
			v_j&\text{if}& m<j\leq k,\end{array}\right.
	\end{eqnarray*}
	and extend it piecewise affine
	linearly to whole of $|{\mathcal K}^{(s)}|= I^m$. Then $f_0$ gets
	modified only on $A$ and  $f(x)=v_0$ for all $x\in \sigma _0$.
	
	Consider an arbitrary $x\in A$.
	Then $x\in \sigma_{j}$ for some $0\le j\le p$. Let $\sigma_{j}=\langle u_{j_0}, u_{j_1}, \ldots, u_{j_m}\rangle$
	and $x=\sum_{i=0}^{m}\alpha_{j_i}u_{j_i}$ with $\sum_{i=0}^{m}\alpha_{j_i}=1$.  Then by \eqref{7}, we have
	\begin{eqnarray*}
		\|f(x)-f_0(x)\|_\infty\le \sum_{i=0}^{m}\alpha_{j_i}\|f(u_{j_i})-f_0(u_{j_i})\|_\infty
		\le \sum_{i=0}^{m}\alpha_{j_i}\mbox{diam}f_0(\Delta)
		<  \frac{\epsilon}{5},
	\end{eqnarray*}
	implying by \eqref{4} that
	\begin{eqnarray*}
		\|f(x)-h(x)\|_\infty\le \|f(x)-f_0(x)\|_\infty+\|f_0(x)-h(x)\|_\infty
		< \frac{\epsilon}{5}+\frac{\epsilon}{6}<\frac{\epsilon}{2}.
	\end{eqnarray*}
	Therefore $f\in B_\epsilon(h)$ and the proof of Step 4 is completed.

	\noindent {\it Step 5. Prove that $f$ has no square roots in
		$\mathcal{C}(I^m)$}.
	
	Since $A\subseteq V$, $f=f_0$ on $A^c$,
	and $f_0(V)\cap V=\emptyset$, we see that  $f(v_0)\neq v_0$.
	Also, by Step 3, we have $f_0^{-1}(\sigma _0)\cap \sigma =\emptyset $ for
	all $\sigma \in {\mathcal K}^{(s)}$ with $\sigma \cap \sigma _0\neq
	\emptyset$, implying that $f=f_0$  on $f_0^{-1}(\sigma _0)$.
	Further, $f_0^{-1}(\sigma_0)$ is infinite, since $\sigma_0$ is infinite and $\sigma_0\subseteq
	R(f_0)^0$. Therefore, as
	$f^{-2}(v_0)\supseteq f^{-1}(f^{-1}(v_0))\supseteq f^{-1}(\sigma
	_0)\supseteq f_0^{-1}(\sigma_0)$, it follows that $f^{-2}(v_0)$ is infinite.  Moreover,  $f|_\sigma$ is injective for every $\sigma \in
	\mathcal{K}^{(s)}$ with $\sigma \neq \sigma _0$  by result  (i) of
	Lemma \ref{L1}, and hence 
	$f^{-1}(x)$ is finite for all $x\ne v_0$ in $I^m$.
	Thus,	$f$ satisfies the conditions of Case  (ii) of
	Theorem \ref{T3}, proving that it has no square roots in $\mathcal{C}(I^m)$.
\end{proof}

\begin{theorem}\label{boundary theorem}
	Let $h\in \mathcal{C}(I^m)$  be such that $h(x_0)=x_0$ for some $x_0 \in \partial I^m$. Then for each $\epsilon
	>0$ the ball $B_{\epsilon }(h)$ of radius $\epsilon$ around $h$ contains a map $f$
	having a continuous iterative square root on $I^m$. In other words,
	the closure of    $\mathcal{W}(2;I^m)$
	contains all maps in $\mathcal{C}(I^m)$ that have a fixed point on the boundary of $I^m$.
\end{theorem}

\begin{proof}
	Let $h\in \mathcal{C}(I^m)$ be such that $h(x_0)=x_0$ for some $x_0\in \partial I^m$. In view of Step 1 of Theorem \ref{Thm1}, without loss of generality, we assume that $h$ is piecewise affine linear on $I^m$.
	Since $h$ is uniformly continuous on $I^m$, there is a $\delta$ with $0<\delta<\frac{\epsilon}{4}$ such that
	\begin{eqnarray}\label{12}
	\|h(x)-h(y)\|_\infty <\frac{\epsilon }{4} ~~\text{whenever}~~x, y\in I^m~~\text{with}~~\|x-y\|_\infty<\delta.
	\end{eqnarray}
	Consider a  triangulation $\mathcal{K}$  of $I^m$ with vertices
	$x_0, x_1, x_2,\ldots, x_r$ such that
	$\|x_{i_0}-x_{i_1}\|_\infty<\frac{\delta}{4}$ for all $2$-simplices
	$\langle x_{i_0},x_{i_1}\rangle$ of $\sigma$ and for all $m$-simplex
	$\sigma \in \mathcal{K}$. Let $y_j:=h(x_j)$ for all $0\le j \le r$ and
	$\sigma_0=\langle x_0, x_1, \ldots, x_m\rangle$ be an $m$-simplex in
	$\mathcal{K}$ having $x_0$ as a vertex. Let $\phi:I^m \to I^m$ be a
	homeomorphism such that $\phi(\sigma_0)=\overline{I^m\setminus
		\sigma_0}$, $\phi\big(\overline{I^m\setminus \sigma_0}\big)=
	\sigma_0$, $\phi={\rm id}$ on $\partial \sigma_0\bigcap
	\partial(I^m\setminus \sigma_0)$, and $\phi^2={\rm id}$ on $I^m$.
	This is where we use the assumption that $x_0$ is on the boundary of
	$I^m.$ Indeed, a little thought shows that there exists a
	piecewise affine linear map $\phi$ that satisfies these conditions.
	Define a function $f_1$ by
	\begin{eqnarray*}
		f_1(x_j) = \left\{ \begin{array}{lll}
			x_j&\text{if}& x_j\in \sigma_0,\\
			x_1&\text{if}& x_j\notin \sigma_0~\text{and}~y_j\in \sigma_0^0,\\
			y_j&\text{if}&x_j\notin \sigma_0 ~\text{and}~y_j\notin \sigma_0^0,
		\end{array}\right.
	\end{eqnarray*}
	and extend it piecewise affine
	linearly to whole of $|{\mathcal K}|= I^m$. Then $f_1\in \mathcal{C}(I^m)$ such that $f_1|_{\sigma_0}={\rm id}$ and $f_1\big(\overline{I^m\setminus \sigma_0}\big)\subseteq \overline{I^m\setminus \sigma_0}$.
	
	\noindent{\bf Claim:} $\rho(f_1,h)<\epsilon$.
	
	Let $x\in I^m$ be arbitrary. If $x\in \sigma_0$, then there exist non-negative reals $\alpha_0, \alpha_1, \ldots, \alpha_m$ with $\sum_{j=0}^{m}\alpha_j=1$ such that $x=\sum_{j=0}^{m}\alpha_jx_j$, implying by \eqref{12} that
	\begin{eqnarray*}
		\|f_1(x)-h(x)\|_\infty&=&\left\|\sum_{j=0}^{m}\alpha_jf_1(x_j)-\sum_{j=0}^{m}\alpha_jh(x_j)\right\|_\infty \\
		&\le & \sum_{j=0}^{m}\alpha_j\left\{\|x_j-x_0\|_\infty+\|h(x_0)-h(x_j)\|_\infty\right\}\\
		&<& \frac{\epsilon }{4}+\frac{\epsilon }{4}=\frac{\epsilon}{2},
	\end{eqnarray*}
	since $\|x_j-x_0\|_\infty \le diam(\sigma_0)<\delta<\frac{\epsilon}{4}$ for all $0\le j\le m$. If $x\notin \sigma_0$, then there exists $\sigma=\langle x_{i_0}, x_{i_1},\ldots, x_{i_m}\rangle\in \mathcal{K}$ such that $x\in \sigma$. Let $x=\sum_{j=0}^{m}\alpha_{i_j}x_{i_j}$, where $\alpha_{i_0}, \alpha_{i_1}, \ldots, \alpha_{i_m}$ are non-negative and $\sum_{j=0}^{m}\alpha_{i_j}=1$. Then by using \eqref{12}, we have
	\begin{eqnarray*}
		\|f_1(x)-h(x)\|_\infty&=&\left\|\sum_{j=0}^{m}\alpha_{i_j}f_1(x_{i_j})-\sum_{j=0}^{m}\alpha_{i_j}h(x_{i_j})\right\|_\infty \\
		&\le & \sum_{j=0}^{m}\alpha_{i_j}\|f_1(x_{i_j})-y_{i_j}\|_\infty\\
		&=&\sum_{\substack{0\le j\le m\\x_{i_j}\in \sigma_0}}^{}\alpha_{i_j}\|f_1(x_{i_j})-y_{i_j}\|_\infty+ \sum_{\substack{0\le j\le m\\ x_{i_j}\notin \sigma_0~\text{and}~ y_{i_j}\in \sigma_0^0} }^{}\alpha_{i_j}\|f_1(x_{i_j})-y_{i_j}\|_\infty\\
		& &+\sum_{\substack{0\le j\le m\\ x_{i_j}\notin \sigma_0~\text{and}~ y_{i_j}\notin \sigma_0^0}}^{}\alpha_{i_j}\|f_1(x_{i_j})-y_{i_j}\|_\infty\\
		& \le &	\sum_{\substack{0\le j\le m\\x_{i_j}\in \sigma_0}}^{}\alpha_{i_j}\|x_{i_j}-y_{i_j}\|_\infty+ \sum_{\substack{0\le j\le m\\ x_{i_j}\notin \sigma_0~\text{and}~ y_{i_j}\in \sigma_0^0} }^{}\alpha_{i_j}\|x_1-y_{i_j}\|_\infty\\
		& &
		+\sum_{\substack{0\le j\le m\\ x_{i_j}\notin \sigma_0~\text{and}~ y_{i_j}\notin \sigma_0^0}}^{}\alpha_{i_j}\|y_{i_j}-y_{i_j}\|_\infty \\
		&\le & \sum_{\substack{0\le j\le m\\x_{i_j}\in \sigma_0}}^{}\alpha_{i_j}\left\{\|x_{i_j}-x_0\|_\infty+\|h(x_0)-h(x_{i_j})\|_\infty\right\}\\
		& &+ \sum_{\substack{0\le j\le m\\ x_{i_j}\notin \sigma_0~\text{and}~ y_{i_j}\in \sigma_0^0} }^{}\alpha_{i_j}\|x_1-y_{i_j}\|_\infty\\
		&<&\frac{\epsilon}{4}+\frac{\epsilon}{4}+\frac{\epsilon}{4}=\frac{3\epsilon}{4},
	\end{eqnarray*}
	since $\|x_{i_j}-x_0\|_\infty \le diam(\sigma_0)<\delta<\frac{\epsilon}{4}$ for all $0\le j\le m$ such that $x_{i_j}\in \sigma_0$, and $\|x_1-y_{i_j}\|_\infty \le diam(\sigma_0)<\delta<\frac{\epsilon}{4}$ for all $0\le j\le m$ such that $x_{i_j}\notin \sigma_0$ and $y_{i_j}\in \sigma_0^0$. This proves the claim.
	
	Now, define a function $f:I^m\to I^m$ by
	\begin{eqnarray*}
		f(x) = \left\{ \begin{array}{cll}
			f_1\circ \phi(x)&\text{if}& x\in \sigma_0,\\
			\phi(x)&\text{if}& x\notin \sigma_0.
		\end{array}\right.
	\end{eqnarray*}
	Since $f_1 \circ \phi (x)=x=\phi(x)$ for all $x\in  \partial \sigma_0\bigcap \partial(I^m\setminus \sigma_0)$, clearly $f$ is continuous on $I^m$.
	We prove that $\rho(f^2,h)<\epsilon$.  Consider an arbitrary $x\in I^m$.
	If $x\in \sigma_0$, then $f(x)\in \overline{I^m\setminus \sigma_0}$, because $\phi(x)\in \overline{I^m\setminus \sigma_0}$ and $f_1\big(\overline{I^m\setminus \sigma_0}\big)\subseteq \overline{I^m\setminus \sigma_0}$. Therefore $f^2(x)\in \sigma_0$, implying
	by \eqref{12} that
	\begin{eqnarray*}
		\|f^2(x)-h(x)\|_\infty
		\le \|f^2(x)-x_0\|_\infty+\|h(x_0)-h(x)\|_\infty
		<\frac{\epsilon}{4}+\frac{\epsilon}{4}=\frac{\epsilon}{2},
	\end{eqnarray*}
	since $\|f^2(x)-x_0\|_\infty\le
	diam(\sigma_0)<\delta<\frac{\epsilon}{4}$. If $x\notin \sigma_0$,
	then $f^2(x)=f_1(x)$, and  therefore $\|f^2(x)-h(x)\|_\infty =
	\|f_1(x)-h(x)\|_{\infty }<\frac{3\epsilon}{4}$. This completes the proof.
\end{proof}

Now it is a natural question to ask what happens if all the fixed
points of the map are in the interior of $I^m$. Theorem 10 of \cite[p.364]{Humke-Laczkovich} shows that
the map $f(x)=1-x$ on the unit interval $[0,1]$, with only $x=\frac{1}{2}$
as the fixed point, cannot be approximated by squares of continuous
maps. In contrast, we have the following.
\begin{example}
	In this example, we show that despite having no continuous square roots, the map $f:I^2\to I^2$ defined by
	\begin{eqnarray*}
		f(x,y) = \left(1-x, \frac{1}{2}\right), \quad \forall (x, y) \in I^2
	\end{eqnarray*}
	having an interior point $(x,y)=(\frac{1}{2},\frac{1}{2})$ as the unique fixed point can be approximated by squares of continuous maps on $I^2$.  To prove that $f$ has no continuous square roots, on the contrary, assume that $f=g^2$ for some $g\in \mathcal{C}(I^2)$. Since $(1-x, \frac{1}{2}) \in R(g)$ for each $x\in I$, there exists $(f_1(x), f_2(x))  \in I^2$ such that $g(f_1(x), f_2(x))=(1-x, \frac{1}{2})$. Then
	\begin{eqnarray*}
		g\left(1-x, \frac{1}{2}\right)=g^2(f_1(x), f_2(x))=f(f_1(x), f_2(x))=\left(1-f_1(x), \frac{1}{2}\right), \quad \forall x\in I,
	\end{eqnarray*}
	implying that
	\begin{eqnarray}\label{1-x}
	g\left(x, \frac{1}{2}\right)=\left(f_1(x), \frac{1}{2}\right), \quad \forall x\in I.
	\end{eqnarray}
	Therefore
	\begin{eqnarray*}
		\left(1-x, \frac{1}{2}\right)=f\left(x, \frac{1}{2}\right)= g^2\left(x, \frac{1}{2}\right)=g\left(f_1(x), \frac{1}{2}\right)=\left(f_1^2(x), \frac{1}{2}\right), \quad \forall x\in I,
	\end{eqnarray*}
	and hence $f_1^2(x)=1-x$ for all $x\in I$. Also, from \eqref{1-x} it follows that $f_1$ is continuous on $I$. Thus, the map $x\mapsto 1-x$ has a continuous square root on $I$, a contradiction to Corollary \ref{C1}. Hence $f$  has no continuous square roots on $I^2$.
	
	Now, to prove that $f$ can be approximated by squares of continuous
	maps,  consider an arbitrary $\epsilon>0$.  Without loss of
	generality, we assume that $\epsilon <\frac{1}{2}.$ The idea is to compress
	$I^2$ to the small strip $I\times [\frac{1}{2}-\frac{\epsilon}{2},
	\frac{1}{2}+\frac{\epsilon}{2}]$ and then rotate this strip 
	ninety  degrees clockwise with a  suitable scaling to stay within the strip.
	More explicitly,  let $g:I^2\to I\times
	[\frac{1}{2}-\frac{\epsilon}{2}, \frac{1}{2}+\frac{\epsilon}{2}]$ be
	defined by $g=g_1\circ g_2$, where $g_1:I\times
	[\frac{1}{2}-\frac{\epsilon}{2}, \frac{1}{2}+\frac{\epsilon}{2}] \to
	I\times [\frac{1}{2}-\frac{\epsilon}{2},
	\frac{1}{2}+\frac{\epsilon}{2}]$ and $g_2: I^2\to I\times
	[\frac{1}{2}-\frac{\epsilon}{2}, \frac{1}{2}+\frac{\epsilon}{2}]$
	are given by
	\begin{eqnarray*}
		g_1(x,y) = \left(\frac{\epsilon-1}{2\epsilon}, \frac{\epsilon+1}{2}\right)+(x,y)\left(\begin{array}{cc}
			0&-\epsilon \\
			\frac{1}{\epsilon} &0
		\end{array}\right)=\left(\frac{1}{2}+\frac{2y-1}{2\epsilon}, \frac{1}{2}- \frac{\epsilon(2x-1)}{2}\right)
	\end{eqnarray*}
	and
	\begin{eqnarray*}
		g_2(x,y) = \left\{ \begin{array}{cll}
			\left(x, \frac{1}{2}-\frac{\epsilon}{2}\right)&\text{if}& (x,y)\in I\times [0, \frac{1}{2}-\frac{\epsilon}{2}],\\
			(x, y)&\text{if} & (x,y)\in I\times [\frac{1}{2}-\frac{\epsilon}{2},\frac{1}{2}+\frac{\epsilon}{2}],\\
			\left(x, \frac{1}{2}+\frac{\epsilon}{2}\right)&\text{if}& (x,y)\in I\times [ \frac{1}{2}+\frac{\epsilon}{2},1].
		\end{array}\right.
	\end{eqnarray*}
	Then $g$ is continuous on $I^2$ such that
	\begin{eqnarray*}
		g^2(x,y) = \left\{ \begin{array}{cll}
			\left(1-x, \frac{1}{2}+\frac{\epsilon}{2}\right)&\text{if}& (x,y)\in I\times [0, \frac{1}{2}-\frac{\epsilon}{2}],\\
			(1-x, 1-y)&\text{if} & (x,y)\in I\times [\frac{1}{2}-\frac{\epsilon}{2},\frac{1}{2}+\frac{\epsilon}{2}],\\
			\left(1-x, \frac{1}{2}-\frac{\epsilon}{2}\right)&\text{if}& (x,y)\in I\times [ \frac{1}{2}+\frac{\epsilon}{2},1],
		\end{array}\right.
	\end{eqnarray*}
	and	it can be easily verified that $\rho(f,g^2)<\epsilon$.
\end{example}

\section{Continuous functions on $\mathbb{R}^m$}\label{sec4}

Consider $\mathbb{R}^m$ in the metric induced by the norm $\|\cdot \|_\infty$ defined as in \eqref{norm}. Since it is a locally compact separable metric space, by Corollary 7.1 of \cite{arens1946}, the compact-open topology on $\mathcal{C}(\mathbb{R}^m)$ is metrizable
with the metric $D$ given by
\begin{eqnarray}\label{D}
	D(f,g)=\sum_{j=1}^{\infty}\mu_j(f,g)
\end{eqnarray}
such that
\begin{eqnarray*}
	\mu_j(f,g)&=&\min\left\{\frac{1}{2^j}, \rho_j(f,g)\right\},\quad \forall j\in \mathbb{N},\\
	\rho_j(f,g)&=&\sup \{\|f(x)-g(x)\|_\infty: x\in M_j\},\quad \forall j\in \mathbb{N},
\end{eqnarray*}
where $(M_j)_{j\in \mathbb{N}}$ is a sequence of compact sets in
$\mathbb{R}^m$ satisfying  that
$\mathbb{R}^m=\bigcup_{j=1}^{\infty}M_{j}$, and
if $M$ is any compact subset of $\mathbb{R}^m$, then $M\subseteq \bigcup_{i=1}^{k}M_{m_i}$ for
some finitely many $M_{m_1}, M_{m_2}, \ldots, M_{m_k}$.
For convenience, we assume that  $M_j$ is an  $m$-simplex of
${\mathbb R}^m$ for each $j\in \mathbb{N}$, and either $M_i\cap M_j=\emptyset$ or
$M_i\cap M_j$ is a common face of both $M_i$ and $M_j$ for all $i,j\in
\mathbb{N}$.  Considering all the faces $M_j$ for all $j\in \mathbb{N}$, we get a countable triangulation of ${\mathbb R}^m.$
Throughout this Section, we fix one such triangulation of ${\mathbb R}^m.$

%
%

\begin{lemma}\label{refinement}
	
	%
	Let $\sigma =\langle x_0, x_1, \ldots, x_m\rangle$ be an $m$-simplex
	of  $\mathbb{R}^m$ and $z\in \sigma$ be such that  $z\ne x_j$ for all $0\le j\le m$. Then there exists a simplicial complex $\mathcal{L}_\sigma$ with the set of vertices $\{x_0, x_1, \ldots, x_m,
	z\}$   such that $| \mathcal{L}_\sigma| =\sigma$.  In other words, every
	$m$-simplex of $\mathbb{R}^m$ can be triangulated to have  a desired
	point
	as an  extra vertex.
\end{lemma}
\begin{proof}
	Without loss of generality, we assume that  $z\in \langle x_0, x_1,
	\ldots, x_k\rangle^0$ for some $1\le k \le m$, say
	$z=\sum_{j=0}^{k}\alpha_jx_j$ for some strictly positive real
	numbers $\alpha_j$ such that $\sum_{j=0}^{k}\alpha_j=1$. Then
	$V_i:=\{x_0, x_1,\ldots, x_{i-1}, z, x_{i+1},\ldots, x_m\}$ is
	geometrically independent for all $0\le i\le k$. In fact, for a
	fixed $0\le i\le k$, if $\beta_j$'s are arbitrary reals such that
	$\sum_{j=0}^{m}\beta_j=0$ and $$\sum\limits_{\substack{0\le j \le m\\
			j\ne i}}^{}\beta_jx_j+\beta_iz=0,$$ then
	\begin{eqnarray*}
		\sum_{\substack{0\le j \le k\\ j\ne i}}^{}(\beta_j+\beta_i\alpha_j)x_j+\beta_i\alpha_ix_i+\sum_{k< j \le m}^{}\beta_jx_j=0,
	\end{eqnarray*}
	and therefore by geometric independence of $\{x_0, x_1, \ldots, x_m\}$ we have $\beta_j+\beta_i\alpha_j=\beta_i\alpha_i=0$ for all $0\le j\le k$ with $j\ne i$, and $\beta_j=0$ for all $k<j\le m$. This implies that $\beta_j=0$ for all $0\le j\le m$, since $\alpha_i>0$.
	
	Now, let
	$\sigma _i:= \langle x_0, x_1, \ldots, x_{i-1}, z, x_{i+1}, \ldots x_m\rangle$
	and ${\mathcal L}_i$ be the simplicial complex of all
	faces of $\sigma _i$ for all $0\le i\le k$.
	
	\noindent {\bf Claim:}
	${\mathcal L}_\sigma= \bigcup
	_{i=0}^k{\mathcal L}_i$ is a triangulation of 
	$\sigma$ with the set of vertices $\{x_0, x_1, \ldots, x_m,
	z\}$. 
	
	Clearly, $\mathcal{L}_\sigma$ is a collection of simplices of $\mathbb{R}^m$ with $\{x_0, x_1, \ldots, x_m,
	z\}$ as its set of vertices. To show that it is a simplicial complex, consider an arbitrary $\lambda \in \mathcal{L}_\sigma$. Then $\lambda \in \mathcal{L}_i$ for some $0\le i\le k$, implying that $\kappa \in \mathcal{L}_i\subseteq \mathcal{L}_\sigma$ whenever $\kappa$ is a face of $\lambda$, because $\mathcal{L}_i$ is a simplicial complex. Next, consider any two simplices $\lambda, \eta \in \mathcal{L}_\sigma$. If $\lambda, \eta \in  \mathcal{L}_i$ for the same $i\in \{0,1,\ldots,k\}$, then clearly $\lambda \cap \eta$ is either empty or a common face of both  $\lambda$ and $\eta$, because $\mathcal{L}_i$ is a simplicial complex. So, let $\lambda \cap \eta \ne \emptyset$, and suppose that $\lambda \in \mathcal{L}_i$ and $\eta \in \mathcal{L}_j$ for some $i\ne j$, where $0\le i, j\le k$.  Further, assume that $\lambda$ and $\eta$ are $s$ and $t$ dimensional, respectively, for some $0\le s,t\le m$. Then it is easily seen $\lambda \cap \eta$ is the $l$-simplex with  vertex set  $V_i\cap V_j$, where $l=\#(V_i\cap V_j)-1$. This implies that $\lambda \cap \eta \in \mathcal{L}_i\cap \mathcal{L}_j$, and therefore it is a common face of both  $\lambda$ and $\eta$. Hence $\mathcal{L}_\sigma$ is a simplicial complex.
	
	Now, to prove that $\mathcal{L}_\sigma$ is a triangulation of $\sigma$, consider an arbitrary $x\in \sigma$. Then $x\in \lambda:=\langle x_{i_0}, x_{i_1}, \ldots, x_{i_s}\rangle$ for some $s$-dimensional face $\lambda$ of $\sigma$, where $0\le s\le m$. If $\lambda\in \mathcal{L}_i$ for some $0\le i\le k$, then clearly $x\in |\mathcal{L}_i|\subseteq  |\mathcal{L}_\sigma|$. If $\lambda\notin \mathcal{L}_i$ for all $0\le i\le k$, then
	there exist at least two $i\in \{0, 1, \ldots, k\}$ such that
	$(V_\lambda\setminus \{z\})\cap (V_i\setminus\{z\})\ne \emptyset$, where $V_\lambda=\{x_{i_0}, x_{i_1}, \ldots, x_{i_s}\}$. This implies that $x\in \lambda \subseteq \cup_{i\in\Gamma}|\mathcal{L}_i|\subseteq |\mathcal{L}_\sigma|$, where $\Gamma:=\{j: 0\le j\le k~\text{and}~(V_\lambda\setminus \{z\})\cap (V_i\setminus\{z\})\ne \emptyset\}$.
	Thus the claim holds and result follows.
\end{proof}

\begin{lemma}\label{tri-Mj}
	Let $\mathbb{R}^m=\bigcup_{j=1}^{\infty}M_j$ be the decomposition of $\mathbb{R}^m$ fixed above, and
	let  $\delta_1, \delta_2, \ldots$ be a sequence of positive real numbers. Then there exists
	a countable triangulation $\mathcal{K}= \bigcup _{j=1}^{\infty}{\mathcal K}_j$ of $\mathbb{R}^m$ such that  ${\mathcal K}_j$ is a triangulation
	of $M_j$ and $diam(\sigma)<\delta_j$ for all $\sigma \in {\mathcal{K}_j}$ and $j\in \mathbb{N}$.
\end{lemma}

\begin{proof}
	Perform barycentric subdivision of the $m$-simplex $M_j$ of sufficient order
	to get a simplicial complex $\mathcal{L}_j$ such that
	$diam(\sigma)<\delta_j$ for all $\sigma \in \mathcal{L}_j$ and $j\in
	\mathbb{N}$. Then $\mathcal{L}:=\bigcup_{j=1}^{\infty}\mathcal{L}_j$
	is not necessarily a simplicial complex, and therefore need not be a
	countable triangulation of $\mathbb{R}^m$, as neighboring $M_j$'s
	might have undergone barycentric subdivisions of different orders.
	Let $\{\sigma _1, \sigma _2, \ldots \}$ be the collection set of all
	$m$-simplices in $\mathcal{L}$ and  $V$ be the set of vertices of
	all simplices in $\mathcal {L}$. Observe that  $V$ may
	contain some points of  $\sigma_i$ that are not its vertices for each $i\in \mathbb{N}$ due to
	the barycentric subdivision of its neighboring $m$-simplices. However, there are at
	most finitely many such points of each $\sigma_i$, since every
	compact subset of ${\mathbb R}^m$ is contained in a finite union of
	$M_j$'s.
	Now, using Lemma \ref{refinement}, we
	refine $\sigma _1$ inductively so that we have a triangulation of $\sigma _1$, which includes all these points as vertices. Next, we
	consider $\sigma _2$ and repeat the same procedure. Note that these refinements add new simplices but do not increase the number of vertices. So, the refinement of $\sigma _2$ does not disturb that of $\sigma _1$ done before. Continuing in this way, we obtain a triangulation $\mathcal{K}_j$ of $M_j$ for each $j\in \mathbb{N}$ so that   $\mathcal{K}=\bigcup_{j=1}^{\infty}\mathcal{K}_j$ is a countable triangulation of ${\mathbb R}^m$ with the set of
	vertices $V$. This completes the proof.
\end{proof}

Having the above lemmas,
we are ready to prove our desired result.

\begin{theorem}\label{Thm2}
	If $h\in \mathcal{C}(\mathbb{R}^m)$, then for each $\epsilon
	>0$ the ball $B_{\epsilon }(h)$ of radius $\epsilon$ around $h$ contains a map $f$
	having no (even discontinuous) iterative square roots on $\mathbb{R}^m$. In particular,
	$\mathcal{W}(2;\mathbb{R}^m)$
	does not contain any ball of $\mathcal{C}(\mathbb{R}^m)$.
\end{theorem}
\begin{proof}
	Let $h\in \mathcal{C}(\mathbb{R}^m)$ and $\epsilon>0$ be arbitrary. 	Our strategy to prove the existence of an  $f\in B_\epsilon(h)$ with no square roots is to make use of
	Theorem \ref{T3}, and we construct this $f$ in a few steps as done in Theorem \ref{Thm1}.
	
	\noindent {\it Step 1: Construct a
		piecewise affine linear map $f_0\in B_\epsilon(h)$ supported by some suitable triangulation $\mathcal{K}$ of
		$\mathbb{R}^m$.}
	
	
	Consider the decomposition
	$\mathbb{R}^m=\bigcup_{j=1}^{\infty}M_{j}$ of $\mathbb{R}^m$ fixed above. Since $h|_{M_j}$ is uniformly continuous
	there exists $\delta_j>0$ such that
	\begin{eqnarray}\label{1'}
		\|h(x)-h(y)\|_\infty <\frac{1}{2^{j+5}}~~\text{whenever}~~x, y\in M_j~~\text{with}~~\|x-y\|_\infty<\delta_j
	\end{eqnarray}
	for each $j\in \mathbb{N}$. Also,
	by Lemma \ref{tri-Mj} there exists a triangulation ${\mathcal{K}}$ of $\mathbb{R}^m$ such that
	%
	$diam(\sigma)<\frac{\delta _j}{4}$  for all 
	$m$-simplex  $\sigma \in{\mathcal{K}}$ contained in $M_j$ and for all  $j\in
	\mathbb{N}$.
	
	Let   $x_0, x_1, \ldots $ be the  vertices of ${\mathcal K} .$
	Choose
	$y_0, y_1, y_2,\ldots$  in $\mathbb{R}^m$ inductively such that $y_j\ne x_j$ and
	\begin{eqnarray}\label{10}
		\|h(x_j)-y_j\|_\infty <\frac{\epsilon }{2^{j+6}}~~\text{for all}~~j\in \mathbb{Z}_+,
	\end{eqnarray}
	and $\{y_{i_0}, y_{i_1}, \ldots, y_{i_m}\}$ is geometrically
	independent whenever $\langle x_{i_0}, x_{i_1},\dots, 
	x_{i_m}\rangle \in {\mathcal{K}}$ for $i_0, i_1,$  $ \ldots,$ $ i_m\in \mathbb{Z}_+$.
	This is possible by Lemma \ref{L5}. Here, while choosing $y_j$'s using Baire Category theorem
	with $y_j\in B_{\frac {\epsilon}{2^{j+6}} }(h(x_j))$
	for all $j\in \mathbb{Z}_+$, satisfying the geometric independence
	condition,  it is also possible to ensure that $y_j\neq x_j$ for all
	$j\in \mathbb{Z}_+$. Then, for each $j\in \mathbb{N}$ we have
	\begin{eqnarray}
		\|y_{i_0}-y_{i_1}\|_\infty&\le& \|y_{i_0}-h(x_{i_0})\|_\infty+\|h(x_{i_0})-h(x_{i_1})\|_\infty+
		\|h(x_{i_1})-y_{i_1}\|_\infty\nonumber \\
		&<&  \frac{\epsilon }{2^{j+5}}+2\frac{\epsilon }{2^{j+6}}=\frac{\epsilon }{2^{j+4}},  \label{8} 
	\end{eqnarray}
	whenever  $\langle x_{i_0},x_{i_1}\rangle$ is a $2$-simplex of
	$\sigma$ for all $m$-simplex $\sigma \in {\mathcal{K}}$ contained in
	$M_j$.  Let $f_0:|{\mathcal{K}}|=\mathbb{R}^m \to \mathbb{R}^m$ be
	the map $f_{U, V}$ defined as in \eqref{f_0}, where $U$ and $V$ are
	the ordered sets $\{x_0, x_1, \ldots\}$ and  $\{y_0, y_1, \ldots\}$,
	respectively. Then $f_0$ is a continuous piecewise affine linear
	self-map of  $\mathbb{R}^m$ satisfying that $f_0(x_j)=y_j$ for all $j\in
	\mathbb{Z}_+$. Also,  by using result  (i) of Lemma \ref{L1}, we
	have $R(f_0)^0\ne \emptyset$, and $f_0|_\sigma$ is injective for all
	$\sigma \in {\mathcal{K}}$ contained in $M_j$ for all $j\in
	\mathbb{N}$. Further,
	\begin{eqnarray}\label{9}
		\mbox{diam}(f_0(\sigma ))<\frac{\epsilon }{2^{j+4}}~~\text{for all}~
		\sigma \in {\mathcal{K}}~\text{contained in}~M_j~\text{and for all}~j\in \mathbb{N}.
	\end{eqnarray}
	
	\noindent{\bf Claim:} $D(f_0,h)<\frac{\epsilon}{4}$.
	
	Consider arbitrary $j\in \mathbb{N}$ and $x\in M_j$. Then $x$ is in
	the interior of an unique $k$-simplex $\langle x_{i_0}, x_{i_1}, \dots,
	x_{i_k}\rangle$  in ${\mathcal{K}}$ for some $0\le k\le m$. Let
	$x=\sum_{j=0}^{k}\alpha_{i_j}x_{i_j}$ with
	$\sum_{j=0}^{k}\alpha_{i_j}=1$.
	Then by \eqref{1'} and \eqref{10}, we have
	\begin{eqnarray}
		\|f_0(x)-h(x)\|_\infty&=&\left\|\sum_{j=0}^{k}\alpha_{i_j}y_{i_j}-\sum_{j=0}^{k}\alpha_{i_j}h(x)\right\|_\infty \nonumber\\
		&\le & \sum_{j=0}^{k}\alpha_{i_j}\left\{\|y_{i_j}-h(x_{i_j})\|_\infty+\|h(x_{i_j})-h(x)\|_\infty\right\}\nonumber\\
		&< & \frac{\epsilon}{2^{j+6}}+\frac{\epsilon}{2^{j+5}}\nonumber \\&<&
		\frac{\epsilon}{2^{j+4}}, \label{11}
	\end{eqnarray}
	since $\|x-x_{i_j}\|_\infty\le diam(\langle x_{i_0}, x_{i_1}, \dots, x_{i_k}\rangle)<\delta_j$ for all $0\le j\le k$. Therefore $\rho_j(f_0, h)\le \frac{\epsilon}{2^{j+4}}<\frac{\epsilon}{2^{j+2}}$ for all $j\in \mathbb{N}$, implying that $D(f_0, h)< \sum_{j=1}^{\infty}\frac{\epsilon}{2^{j+2}}=\frac{\epsilon}{4}$.
	
	\noindent {\it Step 2: Obtain  an $m$-simplex $\Delta \in {\mathcal {K}}$ contained in
		$M_r$ for some $r\in \mathbb{N}$ containing a non-empty open set $Y\subseteq \Delta ^0\cap
		R(f_0)^0$ such that $f_0(Y)\subseteq \Delta _1^0$ for some
		$m$-simplex $\Delta _1$ in ${\mathcal K}$.}
	
	Consider any $m$-simplex $\sigma =\langle x_{i_0}, x_{i_1},\ldots, x_{i_m}\rangle$ in $\mathcal{K}$ contained in $M_r$ for some $r\in \mathbb{N}$. Since $\{y_{i_0}, y_{i_1},\ldots, y_{i_m}\}$ is geometrically
	independent by construction, we have $f_0(U)$ is open in $\mathbb{R}^m$ for each open subset $U$ of
	$\sigma^0$. Choose a sufficiently small  non-empty open subset	$Y_0$ of $f(\sigma )^0$ such that $Y_0\subseteq \Delta^0$ for some
	$m$-simplex $\Delta \in \mathcal {K}$. Then $f_0(Y_0)$ 
	has non-empty interior. Again, choose a  non-empty open subset  $Y$ of $Y_0$ so small
	that $f_0(Y)\subseteq \Delta _1$ for some
	$m$-simplex $\Delta _1\in {\mathcal K}.$ This completes the proof of
	Step 2.


	\noindent {\it Step 3: Show that there exists a barycentric subdivision $\mathcal {K}^{(s)}$ of ${\mathcal {K}}$ for some
		$s \in \mathbb {N}$ with an $m$-simplex  $\sigma_0\in \mathcal{K}^{(s)}$  such that $\sigma_0\subseteq Y$, and $f_0^{-1}(\sigma_0)\cap \sigma=\emptyset$ and
		$f_0(\sigma_0)\cap \sigma=\emptyset$ for all $\sigma\in \mathcal{K}^{(s)}$ with $\sigma\cap\sigma_0\ne
		\emptyset$.}
	
	The proof is similar to that of Step 3 of Theorem \ref{Thm1}.
	
	\noindent {\it Step 4: Modify $f_0$ slightly to get a new piecewise affine
		linear map $f\in B_\epsilon (h)$ that is constant on $\sigma
		_0.$}
	
	The proof is similar to that of Step 4 of Theorem \ref{Thm1}, however we give
	the details here for clarity.  	To retain continuity, when we modify $f_0$ on $\sigma _0$,  we must also modify it on all $m$-simplices in $\mathcal{K}^{(s)}$ adjacent to it. 
	Let $A:=\bigcup_{j=0}^{p}\sigma_j$,
	where $\sigma_j$ is precisely an $m$-simplex in
	$\mathcal{K}^{(s)}$ such that $\sigma_j\cap \sigma_0\ne \emptyset$ for all $0\le j\le p$.
	Then, clearly $A\subseteq V\subseteq \Delta
	^0$.
	
	Let the vertices of $\sigma _0$ be  $u_0,\ldots, u_{m-1}$ and $u_m$, and
	let
	$v_j:=f_0(u_j)$ for all $0\le j \le m$.  Then  $v_j\ne u_j$ for all $0\le j\le m$, because $u_j\in V$ for all $0\le j\le m$ and $f_0(V)\cap V=\emptyset$. Also,
	as $u_0, u_1, \ldots , u_m$ are geometrically
	independent vectors contained in $\Delta $, by result  (ii) of Corollary \ref{C3}, we see that $v_0, \ldots , v_m$ are geometrically independent. Further, by Step 2,
	they are all contained in the interior of a single simplex $\Delta
	_1$ of ${\mathcal{K}}$. Hence, by result  (iii) of Corollary \ref{C3}, we have $f_0(v_j)\neq v_j$ for some $0\le j\le m$. Without loss of generality, we assume that $f_0(v_0)\neq
	v_0.$

	Extend $\{u_0, \ldots, u_m\}$ to  $\{u_0, u_1, \ldots\}$   to include the vertices of all the simplices of ${\mathcal K}^{(s)}$ and let $v_j:=
	f_0(u_j)$ for all $j \in \mathbb{Z}_+$. Now, define $f$ by
	\begin{eqnarray*}
		f(u_j) = \left\{ \begin{array}{lll}
			v_0&\text{if}& 0\leq j\leq m,\\
			v_j&\text{if}& j>m,\end{array}\right.
	\end{eqnarray*}
	and extend it countably piecewise affine
	linearly to whole of $|{\mathcal K}^{(s)}|= \mathbb{R}^m$. Then $f_0$ gets
	modified only on $A$ and  $f(x)=v_0$ for all $x\in \sigma _0$.
	
	Consider any arbitrary $x\in A$.
	Then $x\in \sigma_{j}$ for some $0\le j\le p$. Let $\sigma_{j}=\langle u_{j_0}, u_{j_1}, \ldots, u_{j_m}\rangle$
	and $x=\sum_{i=0}^{m}\alpha_{j_i}u_{j_i}$ with $\sum_{i=0}^{m}\alpha_{j_i}=1$.
	Then by \eqref{9}, we have
	\begin{eqnarray*}
		\|f(x)-f_0(x)\|_\infty\le \sum_{i=0}^{m}\alpha_{j_i}\|f(u_{j_i})-f_0(u_{j_i})\|_\infty
		\le \sum_{i=0}^{m}\alpha_{j_i}\mbox{diam}f_0(\Delta )
		<  \frac{\epsilon}{2^{r+4}},
	\end{eqnarray*}
	implying by \eqref{11} that
	\begin{eqnarray*}
		\|f(x)-h(x)\|_\infty\le \|f(x)-f_0(x)\|_\infty+\|f_0(x)-h(x)\|_\infty
		< \frac{\epsilon}{2^{r+4}}+\frac{\epsilon}{2^{r+4}}<\frac{\epsilon}{2^{r+3}}.
	\end{eqnarray*}
	Therefore $\rho_r(f, h)< \frac{\epsilon}{2^{r+2}}$. Also, we have $\rho_j(f, h)=\rho_j(f_0,h)<\frac{\epsilon}{2^{j+2}}$ for $j\ne r$. Hence $D(f, h)< \sum_{j=1}^{\infty}\frac{\epsilon}{2^{j+2}}=\frac{\epsilon}{4}$, proving that $f\in B_\epsilon(h)$.
	
	\noindent {\it Step 5. Prove that $f$ has no square
		roots in $\mathcal{C}(\mathbb{R}^m)$.}
	
	It is clear that $f(v_0)\neq v_0$, because $A\subseteq V$,  $f=f_0$ on $A^c$ and $f_0(V)\cap V=\emptyset$.
	Also, by Step
	3, we have $f_0^{-1}(\sigma _0)\cap \sigma =\emptyset $ for all $\sigma \in
	{\mathcal K}^{(s)}$ with $\sigma \cap \sigma _0\neq \emptyset$, implying that $f=f_0$  on $f_0^{-1}(\sigma _0)$. Further, $f_0^{-1}(\sigma_0)$ is uncountable, since $\sigma_0$ is uncountable and $\sigma_0\subseteq R(f_0)^0$.
	Therefore, as $f^{-2}(v_0)\supseteq f^{-1}(f^{-1}(v_0))\supseteq
	f^{-1}(\sigma _0)\supseteq f_0^{-1}(\sigma_0)$, it follows that $f^{-2}(v_0)$ is uncountable.
	Moreover, $f|_\sigma$ is
	injective for every $\sigma \in \mathcal{K}^{(s)}$ with $\sigma \neq
	\sigma _0$  by result  (i) of Lemma \ref{L1}, and hence   $f^{-1}(x)$ is countable for all $x\ne
	v_0$ in $\mathbb{R}^m$.
	Thus,
	$f$ satisfies the hypothesis of Case  (iii) of
	Theorem \ref{T3}, proving that $f$ has no square roots in $\mathcal{C}(\mathbb{R}^m)$. 
\end{proof}

We have seen in  Theorem \ref{Thm1} (resp. \ref{Thm2})  that each open
neighborhood of each map in  $\mathcal{C}(I^m)$ (resp. in
$\mathcal{C}(\mathbb{R}^m)$) has a continuous map which
does not have even discontinuous square roots. Additionally, if $X$ and $Y$  are
locally compact Hausdorff spaces,  and $\phi:X\to Y$ is a
homeomorphism, then the map $f\mapsto  \phi \circ f\circ \phi ^{-1}$ is a
homeomorphism of $\mathcal{C}(X)$  onto $\mathcal{C}(Y)$, where both $\mathcal{C}(X)$  and $\mathcal{C}(Y)$ have the compact-open topology, and moreover
$f=g^2$ on $X$ if and only if $\phi \circ f\circ \phi ^{-1} = (\phi \circ
g\circ \phi ^{-1})^2$ on $Y$. Hence it follows that Theorem \ref{Thm1} (resp. Theorem \ref{Thm2}) is true
if $I^m$ (resp.  $\mathbb{R}^m$) is replaced by any topological space
homeomorphic to it. For the same reason, Theorem \ref{boundary
	theorem} is also true if $I^m$ is replaced by any topological space
homeomorphic to it.

\section{$L^p$ denseness of iterative squares in $\mathcal{C}(I^m)$}\label{sec5}
In this section, we prove that the iterative squares of continuous self-maps of $I^m$ are $L^p$ dense in $\mathcal{C}(I^m)$.
Given any nontrivial compact subintervals $I_1, I_2, \ldots, I_m$ of $I$, let $I_1\times I_2\times \cdots \times I_m$ denote the closed rectangular region in $I^m$ defined by  $$I_1\times I_2\times \cdots \times I_m=\{(x_1,x_2, \ldots, x_m): x_i\in I_i ~\text{for}~1\le i \le m\}.$$
Consider $\mathcal{C}(I^m)$ in the  $L^p$ norm defined by
\begin{eqnarray*}
	\|f\|_p=\left(\int_{I^m}^{} \|f(x)\|_{\infty}^pd\mu(x) \right)^\frac{1}{p},
\end{eqnarray*}
where $\mu $ is the Lebesgue measure on ${\mathbb R}^m$.
First, we prove a result which gives a sufficient condition for extending a given function to a square.
\begin{theorem}\label{T4}
	Let $K$ be a proper closed subset of $I^m$ and $f:K\to I^m$ be continuous. Then there exists a $g\in \mathcal{C}(I^m)$ such that $f=g^2|_K$.
\end{theorem}
\begin{proof}
	Let $I_1\times I_2\times \cdots \times I_m$ be a closed rectangular
	region in $I^m\setminus K$ and $\tilde{g}$ be a homeomorphism of $K$
	into $I_1\times I_2\times \cdots \times I_m$. Extend $\tilde{g}$ to
	$K\cup R(\tilde{g})$ by setting $g(x)=f(\tilde{g}^{-1}(x))$ for
	$x\in R(\tilde{g})$. Then $g$ is continuous on the closed set $K\cup
	R(\tilde{g})$, and therefore by  an extension of Tietze's extension  theorem
	by Dugundji (see Theorem 4.1 and Corollary 4.2 of \cite[pp.357-358]{dugundji}) it has an
	extension to a continuous self-map of $I^m$, which also we denote by $g$. Further,
	for each $x\in K$, we have
	\begin{eqnarray*}
		g^2(x)&=&g(g(x))\\
		&=&g(\tilde{g}(x))\\
		&=&f(\tilde{g}^{-1}(\tilde{g}(x)))~~(\text{since}~\tilde{g}(x)\in R(\tilde{g}))\\
		&=&f(x),
	\end{eqnarray*}
	implying that $f=g^2|_K$.
\end{proof}
It is worth noting that the assumption in the above theorem that $K$ is a proper subset of $I^m$ cannot be dropped. Indeed, if $K=I$, then the continuous map $f(x)=1-x$ on $K$ has no square roots in $\mathcal{C}(I)$ (see \cite[pp.425-426]{kuczma1990} or Corollary \ref{C1}). Such maps can be constructed on $K=I^m$ for general $m$ using Theorem \ref{connected components}.
We now prove our desired result.
\begin{theorem}
	$\mathcal{W}(2;I^m)$
	is $L^p$ dense in $\mathcal{C}(I^m)$.
\end{theorem}

\begin{proof}
	Consider an arbitrary $f\in \mathcal{C}(I^m)$. Then, by Theorem \ref{T4}, for each $\epsilon>0$ there exists a map $g_\epsilon \in \mathcal{C}(I^m)$ such that $f|_{I_\epsilon\times I_\epsilon\times \cdots \times I_\epsilon}=g_\epsilon^2|_{I_\epsilon\times I_\epsilon\times \cdots \times I_\epsilon}$, where $I_\epsilon=[\epsilon, 1]$. Now,
	%
	%
	\begin{eqnarray*}
		\|g_\epsilon^2-f\|_p^p&=&\int_{I^m}^{} \|g_\epsilon^2(x)-f(x)\|_{\infty }^pd\mu(x)\\
		&=&\int_{({I_\epsilon\times I_\epsilon\times \cdots \times I_\epsilon})^c}^{} \|g_\epsilon^2(x)-f(x)\|_{\infty}^pd\mu(x)\\
		&\le&  \mu(({I_\epsilon\times I_\epsilon\times \cdots \times I_\epsilon})^c)\\
		&=&1-(1-\epsilon)^m,
	\end{eqnarray*}
	implying that $\|g_\epsilon^2-f\|_p \to 0$ as $\epsilon \to 0$. Hence $\lim\limits_{\epsilon \to 0}g_\epsilon^2=f$ in $L^p$, and the result follows.
\end{proof}


\noindent {\bf Acknowledgments:} 	The authors are grateful to the referee for his/her many helpful comments.
The first author is supported by J
C Bose Fellowship of the Science and Engineering Board, India. The second author is
supported by the Indian Statistical Institute,  Bangalore 
through J C Bose Fellowship of the
first author, and by the National Board for Higher Mathematics, India through No:
0204/3/2021/R\&D-II/7389.




\begin{thebibliography}{9}
	\bibitem{Abel}
	N. H. Abel, Oeuvres Complètes, T.II, {\it Christiania}, (1881), 36--39.
	
	\bibitem{arens1946} R. F. Arens,
	A topology for spaces of transformations, {\it Ann. of Math.},  47 (1946), 3, 480--495.
	
	\bibitem{Babbage1815}
	C. Babbage, Essay towards the calculus of functions, {\it Philos. Trans.}, (1815), 389--423.
	
	\bibitem{Baron-Jarczyk}
	K. Baron, W. Jarczyk, Recent results on functional equations in a single variable,
	perspectives and open problems, {\it Aequationes Math.},  61 (2001), 1--48.
	
	\bibitem{BDR}
	B. V. R. Bhat, S. De, N. Rakshit, A caricature of
	dilation theory, {\it Adv. Oper. Theory}, 6 (2021), 4, Paper No. 63, 20 pp. 
	
	\bibitem{blokh-coven}
	A. Blokh, E. Coven, M. Misiurewicz, Z. Nitecki, Roots of continuous piecewise monotone
	maps of an interval, {\it Acta Math. Univ. Comenian. (N.S.)},  60 (1991), 3--10.
	
	\bibitem{Blokh}
	A. M. Blokh, The set of all iterates is nowhere dense in
	{$C([0,1],[0,1])$}, {\it Trans. Amer. Math. Soc.},  333 (1992), 2,  787--798.
	
	\bibitem{Bodewadt}
	U. T. B\"{o}dewadt,  Zur Iteration reeller funktionen, {\it Math. Z.}, 49  (1944), 497--516.
	
	\bibitem{Bogatyi}
	S. Bogatyi, On the nonexistence of iterative roots, {\it Topology Appl.}, 76
	(1997), 97--123.
	
	\bibitem{deo}
	S. Deo, {\it Algebraic topology}, Hindustan Book Agency, New Delhi, 2018.
	
	\bibitem{dugundji}
	J. Dugundji, An extension of Tietze's theorem, {\it Pacific J. Math.},  1 (1951), 353--367.
	
	\bibitem{Edgar}
	G. A. Edgar, Fractional iteration of series and transseries, {\it Trans. Amer. Math. Soc.},  365 (2013), 11,  5805--5832.
	
	\bibitem{fort1955}
	M. K. Fort Jr, The embedding of homeomorphisms in flows, {\it Proc. Amer. Math. Soc.},  6 (1955),
	960--967.
	
	\bibitem{Humke-Laczkovich1989}
	P. D. Humke, M. Laczkovich, The Borel structure of iterates of continuous functions, {\it Proc. Edinburgh Math. Soc.},  32 (1989), 483--494.
	
	\bibitem{Humke-Laczkovich}
	P. D. Humke, M. Laczkovich, Approximations of continuous functions by squares, {\it Ergodic Theory Dynam. Systems}, 10 (1990), 2, 361--366.
	
	\bibitem{Iannella}
	N. Iannella and L. Kindermann, Finding iterative roots with a spiking
	neural network, {\it Inform. Process. Lett.}, 95 (2005), 545--551.
	
	\bibitem{Issacs}
	R. Isaacs, Iterates of fractional order, {\it Canad. J. Math.},  2 (1950), 409--416.
	
	\bibitem{Jarczyk-Zhang}
	W. Jarczyk, W. Zhang, Also set-valued functions do not like
	iterative roots, {\it Elemente Math.}, 62 (2007),  1--8.
	
	\bibitem{Kindermann}
	L. Kindermann, Computing iterative roots with neural networks,
	{\it Proc. Fifth Conf. Neural Info. Processing} 2 (1998), 713--715.
	
	\bibitem{Konigs}
	G. K\"{o}nigs, Recherches sur les int\'{e}grales de certaines \'{e}quations fonctionnelles,
	{\it Ann. Ecole Norm. Sup.} (3) (1884), Suppl.  1, 3--41.
	
	\bibitem{kuczma1961}
	M. Kuczma, On the functional equation $\phi^n(x) = g(x)$, {\it Ann. Polon. Math.},  11 (1961), 161--175.
	
	\bibitem{Kuczma1968}
	M. Kuczma, {\it Functional Equations in a Single Variable}, Polish Scientific, Warsaw, 1968.
	
	\bibitem{Kuczzma1969}
	M. Kuczma, Fractional iteration of differentiable functions, {\it Ann. Pol. Math.}, 22 (1969/70), 217--227.
	
	\bibitem{kuczma1990}
	M. Kuczma, B. Choczewski, R. Ger,  {\it Iterative functional equations}, volume
	32 of Encyclopedia of Mathematics and its Applications. Cambridge University
	Press, Cambridge, 1990.
	
	\bibitem{Lesniak2002}
	Z. Le\'{s}niak, On fractional iterates of a homeomorphism of the plane, {\it Ann. Polon. Math.},   79 (2002), 2, 129--137.
	
	\bibitem{Lesniak2010}
	Z. Le\'{s}niak, On fractional iterates of a Brouwer homeomorphism embeddable in a flow, {\it J. Math.
		Anal. Appl.},  366 (2010), 310--318.
	
	
	
	\bibitem{li2009}
	L. Li, J. Jarczyk, W. Jarczyk, W. Zhang, Iterative roots of mappings with a unique set-value point,
	{\it Publ. Math. Debrecen},  75 (2009), 203--220.
	
	\bibitem{LiYangZhang2008}
	L. Li, D. Yang, W. Zhang, A note on iterative roots of PM functions, {\it J. Math. Anal. Appl.},  341  (2008), 2, 1482--1486.
	
	\bibitem{li2012}
	L. Li, W. Zhang, Construction of usc solutions for a multivalued iterative equation of order
	$n$, {\it Results Math.},  62 (2012), 203--216.
	
	\bibitem{LiZhang2018}
	L. Li, W. Zhang, The question on characteristic endpoints for iterative roots of PM functions, {\it J. Math. Anal. Appl.},   458  (2018), 1,  265--280.
	
	\bibitem{Lin2014}
	Y. Lin, Existence of iterative roots for the sickle-like functions, {\it J.
		Inequal. Appl.}, (2014), Paper No. 204, 23 pp.
	
	\bibitem{Lin-Zeng-Zhang2017}
	Y. Lin, Y. Zeng, W. Zhang, Iterative roots of clenched single-plateau
	functions, {\it Results Math.}, 71 (2017),  1-2, 15--43.
	
	\bibitem{LiuJarczykZhang2012}
	L. Liu, W. Jarczyk, L. Li, W. Zhang, Iterative roots of piecewise monotonic functions of nonmonotonicity
	height not less than $2$, {\it Nonlinear Anal.},   75 (2012), 1,  286--303.
	
	\bibitem{LiuLiZhang2018}
	L. Liu, L. Li, W. Zhang, Open question on lower order iterative roots for PM functions, {\it J. Difference Equ. Appl.},   24    (2018), 5, 825--847.
	
	\bibitem{LiuLiZhang2021}
	L. Liu, L. Li, W. Zhang, Iterative roots of exclusive multifunctions,
	{\it J. Difference Equ. Appl.} 27 (2021),  1, 41--60.
	
	\bibitem{LiuZhang2011}
	L. Liu,  W. Zhang, Non-monotonic iterative roots extended from characteristic intervals, {\it J.
		Math. Anal. Appl.}, 378 (2011), 359--373.
	
	\bibitem{Liu-zhang2021}
	L. Liu, W. Zhang, Genetics of iterative roots for PM functions, {\it Discrete Contin. Dyn. Syst.}, 41 (2021),  5, 2391--2409.
	
	\bibitem{Martin}
	R. J. Martin, M\"{o}bius splines are closed under continuous iteration, {\it Aequat.
		Math.}, 64 (2002),  274--296.
	
	
	\bibitem{narayaninsamy2000}
	T. Narayaninsamy, A connection between fractional iteration and graph theory, {\it  Appl. Math. Comput.},  107  (2000), 2-3, 181--202.
	
	\bibitem{rice1980}
	R. E. Rice, B. Schweizer, A. Sklar, When is $f(f(z)) = az^2 + bz + c?$, {\it Amer. Math. Monthly},  87 (1980), 4,  252--263.
	
	
	\bibitem{simon1989}
	K. Simon, Some dual statements concerning Wiener measure and Baire category, {\it Proc. Amer. Math. Soc.},  106 (1989), 2, 455--463.
	
	\bibitem{simon1990}
	K. Simon, Typical functions are not iterates, {\it Acta Math. Hungar.},  55 (1990), 133--134.
	
	\bibitem{simon1991a}
	K. Simon, The set of second iterates is nowhere dense in $C$, {\it Proc. Amer. Math. Soc.},  111 (1991), 1141--1150.
	
	\bibitem{simon1991b}
	K. Simon, The iterates are not dense in $C$, {\it Math. Pannon.}, 2 (1991), 71--76.
	
	\bibitem{solarz2003}
	P. Solarz, On some iterative roots on the circle, {\it Publ. Math. Debrecen},  63 (2003), 677--692.
	
	\bibitem{solarz2006}
	P. Solarz, Iterative roots of some homeomorphisms with a rational rotation number, {\it Aequationes
		Math.},  72 (2006), 152--171.
	
	\bibitem{Targonski1981}
	G. Targonski, {\it Topics in Iteration Theory}, Vandenhoeck and Ruprecht, G\"{o}ttingen, 1981.
	
	\bibitem{Targonski1995}
	G. Targonski, Progress of iteration theory since 1981, {\it Aequationes Math.},  50 (1995),  50--72.
	
	\bibitem{wilf2006}
	H. S. Wilf, {\it generatingfunctionology}, Third edition, A K Peters, Ltd., Wellesley, MA, 2006.
	
	\bibitem{Xu-zhang}
	B. Xu, W. Zhang,
	Construction of continuous solutions and stability for the polynomial-like iterative equation,
	{\it J. Math. Anal. Appl.}, 325 (2007), 1160--1170.
	
	\bibitem{Yu2021}
	Z. Yu, L. Li, L. Liu, Topological classifications for a class of 2-dimensional quadratic mappings and an application to iterative roots,
	{\it Qual. Theory Dyn. Syst.}, 20 (2021), 1, Paper No. 2, 25 pp.
	
	\bibitem{zdun2000}
	M. C. Zdun, On iterative roots of homeomorphisms of the circle, {\it Bull. Polish Acad. Sci. Math.},  48 (2000), 2, 203--213.
	
	\bibitem{zdun2014}
	M. C. Zdun, On approximative embeddability of diffeomorphisms in $C^1$-flows, {\it J. Differ. Equ.
		Appl.},  20 (2014), 1427--1436.
	
	\bibitem{zdun-soalrz}
	M. C. Zdun, P. Solarz, Recent results on iteration theory: iteration groups and semigroups in
	the real case, {\it Aequationes Math.},  87 (2014), 201--245.
	
	\bibitem{zdun-zhang}
	M. C. Zdun, W. Zhang, Koenigs embedding flow problem
	with global $C^1$ smoothness, {\it J. Math. Anal. Appl.}, 374 (2011), 633--643.
	
	\bibitem{zhang-zeng}
	W. Zhang, Y. Zeng, W. Jarczyk, W. Zhang, Local $C^1$ stability versus global $C^1$
	unstability for iterative roots, {\it J. Math. Anal. Appl.}, 386  (2012), 75--82. 
	
	\bibitem{Zhang-zhang2007}
	W. Zhang, W. Zhang, Computing iterative roots of polygonal functions, {\it J. Comput. Appl. Math.}, 205 (2007), 1, 497--508.
	
	\bibitem{zhang-zhang}
	W. Zhang, W. Zhang,  Continuity of iteration and approximation of iterative roots, {\it J.
		Comput. Appl. Math.} 235 (2011), 1232--1244.
	
	\bibitem{zhang1995}
	W. Zhang, A generic property of globally smooth iterative roots, {\it Sci. China Ser.
		A.}, 38 (1995),  267--272.
	
	\bibitem{wzhang1997}
	W. Zhang, PM functions, their characteristic intervals and iterative roots. {\it Ann. Polon.
		Math.}, 65 (1997), 2, 119--128.	
	
	
	
	%
	%
	%
	%
	
	
\end{thebibliography}
\end{document}